\documentclass[11pt]{amsart}

\usepackage{graphicx}
\usepackage{amsmath,amsthm}
\usepackage{amssymb,amsfonts,amsxtra,dsfont}
\usepackage{tikz}
\usetikzlibrary{matrix,arrows}

\theoremstyle{plain}
\newtheorem{theorem}{Theorem}[section]

\newtheorem{prop}[theorem]{Proposition}
\newtheorem{lemma}[theorem]{Lemma}
\newtheorem{corollary}[theorem]{Corollary}
\newtheorem*{mainthm}{Main Theorem}

\theoremstyle{definition}

\newtheorem{remark}[theorem]{Remark}

\newcommand{\fat}[1]{\mathds{#1}}
\newcommand{\hsm}[1]{\mathcal{#1}}

\newcommand{\Z}{\fat{Z}}
\newcommand{\ZZ}{\Z}
\newcommand{\R}{\fat{R}}
\newcommand{\RR}{\R}

\newcommand{\NN}{\fat{N}}
\newcommand{\NNhat}{\widehat{\NN}}
\newcommand{\minus}{\setminus}

\newcommand{\link}[2]{\mathrm{Lk}\!\left(#1,#2\right)}
\newcommand{\uplink}[2]{\mathrm{Lk}_{\uparrow}(#1,#2)}
\newcommand{\dnlink}[2]{\mathrm{Lk}_{\downarrow}(#1,#2)}

\newcommand{\st}[2]{\mathrm{St}\!\left(#1,#2\right)}
\newcommand{\inv}{^{^{-1}}}
\newcommand{\term}{\partial_{\!_{\mathbf{+}}}\!}
\newcommand{\init}{\partial_{\!_{\mathbf{-}}}\!}
\newcommand{\Cay}[1]{\mathit{Cay}\!\left(#1\right)}
\newcommand{\PP}[1]{\fat{P}_{#1}}
\newcommand{\tilPP}[1]{\widetilde{\fat{P}}_{#1}}
\newcommand{\area}[1]{\mathrm{Area}\!\left(#1\right)}
\newcommand{\res}[1]{\left|_{#1}\right.}
\newcommand{\pres}[2]{\left\langle #1\,\Big{|}\,#2\right\rangle}
\newcommand{\vrim}[1]{V\mathrm{Rim}\!\left(#1\right)}
\newcommand{\erim}[1]{E\mathrm{Rim}\!\left(#1\right)}
\newcommand{\cell}[1]{C_{#1}}
\newcommand{\alt}[3]{\mathrm{Alt}_{#1}\left(#2,#3\right)}
\newcommand{\tri}[2]{\Delta_{#1}\left(#2\right)}
\newcommand{\zl}{\widetilde{h}\inv(0)}
\newcommand{\bd}{\partial}
\newcommand{\ep}{\hfill$\square$}
\newcommand{\join}{\star}

\begin{document}
\title[Subgroups of perturbed RAAGs]{Dehn functions and finiteness properties of subgroups of perturbed right-angled Artin groups}

\author[N.~Brady]{Noel Brady$^1$}
\address{Dept.\ of Mathematics\\
        University of Oklahoma\\
	Norman, OK 73019}
\email{nbrady@math.ou.edu}

\author[D.\ P.~Guralnik]{Dan P.\ Guralnik}
\address{Dept.\ of Mathematics\\
        University of Oklahoma\\
	Norman, OK 73019}
\email{dan.guralnik@ou.edu}

\author[SR.~Lee]{Sang Rae Lee}
\address{Dept.\ of Mathematics\\
        University of Oklahoma\\
	Norman, OK 73019}
\email{srlee@math.ou.edu}

\date{\today}

\begin{abstract}
We introduce the class of perturbed right-angled Artin groups. These are constructed  by gluing Bieri double groups into
standard right-angled Artin groups.
As a first application of this construction, we obtain families of CAT(0) groups which contain finitely presented subgroups
which are not of type FP$_3$ and have exponential, or polynomial Dehn functions of prescribed degree.
	
\end{abstract}

\maketitle
\footnotetext[1]{N.\ Brady was partially  supported by NSF grant no.\  DMS-0906962}

Right-angled Artin groups (RAAGs) have been the subject of intense study over the past two decades: 
they have non-positively curved cubical $K(\pi, 1)$ spaces \cite{CharneyDavis};  
they embed (as finite index subgroups) in right-angled Coxeter groups \cite{DavisJan}; 
they have subgroups with interesting finiteness properties \cite{[Bestvina-Brady], Leary-Nucinkis, Hsu-Leary}; 
they admit embeddings of special cube complex groups \cite{HWise};  there is an ongoing investigation 
of their (co)homology groups and asymptotic topology \cite{DavisLeary, JenM, BrMe, LyM, DO}; various filling invariants for these groups and special subgroups are being determined \cite{DERY, [Dyson], [ABDDY], BeCh};  
a start has been made on their quasi-isometric classification \cite{Bestvina-Kleiner-Sageev, BJNeumann}, and 
on understanding their  automorphism groups \cite{BCV, CharVogt, CCV}.

In this paper we introduce the class of {\em perturbed RAAGs} and  investigate the geometry and topology of their kernel subgroups.
First we recall  what is known about kernels of  RAAGs; see section~\ref{back} for more details.

Each finite graph $\Gamma$ determines a finite presentation of a RAAG $A_\Gamma$.
The vertices of $\Gamma$ are in 1-to-1 correspondence with generators of $A_\Gamma$, and
two generators commute precisely when the corresponding vertices are adjacent in $\Gamma$.
The group $A_\Gamma$ is the fundamental group of a finite, non-positively curved, piecewise euclidean, cubical
complex $X_\Gamma$.
There is an epimorphism $h: A_\Gamma \to \Z$, obtained by sending all generators of $A_\Gamma$ to
a generator   of $\Z$.  Depending on $\Gamma$, the kernel, $\ker(h)$, may have very interesting geometry and
topology. We summarize the known results.

\begin{enumerate}
\item  {\em Topology of $\ker(h)$.} The finiteness properties of $\ker(h)$ are determined by the topology of
the flag complex, $K_\Gamma$,  determined by $\Gamma$.  There is a real-valued, $h$--equivariant Morse function on
the universal cover of the cubical complex $X_\Gamma$. The kernel subgroup $\ker(h)$ acts properly and cocompactly
on the level sets of this Morse function, so the finiteness properties are determined by the topology of the level
sets. This topology is determined by the topology of the {\em ascending and descending links} of the
Morse function. In this case, the ascending and descending links are all isomorphic to the flag complex $K_\Gamma$.
For example, if $K_\Gamma$ is $1$--connected, then $\ker(h)$ is finitely presented. 

 \item {\em Geometry of $\ker(h)$.} In the case that $K_\Gamma$ is $1$--connected, we can ask about the
 {\em Dehn function} of the finitely presented kernel subgroup $\ker(h)$. Will Dison \cite{[Dyson]} proved that the
 Dehn function of $\ker(h)$ is bounded above by the polynomial function $x^4$.
 \end{enumerate}

A key ingredient in the construction of perturbed RAAGs is the Bieri doubling procedure.
In  \cite{[Bieri]}, Bieri showed that Stallings' example \cite{[Stallings]} of a finitely presented group
 with non-finitely generated third homology is a subgroup of the direct product of three copies of the free group $F_2$. 
 This is a RAAG whose defining flag complex $K_\Gamma$ is the $3$--fold join of $0$--spheres. 
 In \cite{BBMS} Bieri's construction was used to produce finitely presented subgroups of (bi)-automatic
 groups with exponential Dehn functions or with polynomial Dehn functions of prescribed degree. 
  In \cite{[Brady-Blue-Book]}  explicit Bieri double groups are constructed as kernels of the fundamental groups of
 non-positively curved, piecewise-euclidean cubical complexes. The basic objects to construct are non-positively curved
 squared complexes whose fundamental groups are $F_n \rtimes \Z$ with polynomially or exponentially distorted fibers $F_n$.
 See Section 2 for more details. The topology of the Bieri doubles is well understood. They are doubles of $F_n \rtimes \Z$ over 
 the $F_n$ fiber, and so have finite,  $2$--dimensional $K(\pi,1)$ spaces.

 Here is an overview of the definition of a perturbed RAAG.  See Section~3 for details. 
 Start with a finite graph $\Gamma$, and consider the corresponding RAAG $A_\Gamma$.
 An edge of $\Gamma$ corresponds to a $\Z^2$ subgroup of $A_\Gamma$, which is the fundamental group of a square torus
in the cubical complex $X_\Gamma$. In order to define a {\em perturbed RAAG}, one selects certain edges of $\Gamma$, and replaces
their corresponding $2$--tori by suitably chosen non-positively curved squared complexes with fundamental groups $F_n\rtimes \Z$.  
The latter are chosen very 
carefully to ensure  that the  resulting 
perturbed RAAG is non-poitively curved cubical, and admits an epimorphism $h$  to $\Z$ whose kernel
$\ker(h)$ has the same finiteness properties as the original RAAG. However, the geometry of $\ker(h)$ will typically be
dominated by the geometry of the Bieri double.  The following table summarizes the situation.

  \bigskip

 \begin{center}
{\tiny 
\begin{tabular}{|c||c|c|c|}
 \hline
 \rule{0in}{2.5ex}{\bf CAT(0) }&  {\bf Asc/Desc}  &  {\bf Finiteness}  &  {\bf Dehn function} \\
     {\bf Group}  &  {\bf Links  }   			      &   {\bf properties of $\ker(h)$}                       &  {\bf of $\ker(h)$ } \\ [0.5ex]
\hline
\hline
 \rule{0in}{2.5ex}{\bf RAAG} &  Isomorphic     &   Determined by  the   &  Has upper bound of $x^4$. \\
                     & to $K_\Gamma$.  & topology of $K_\Gamma$.  &   Not known in general.\\ [0.5ex]
 \hline
 \rule{0in}{2.5ex}{\bf Bieri}    & Contractible.   & Has a finite  & $x^n$ or $e^x$.   \\
 {\bf double}  &                       &  $2$--dimensional  $K(G,1)$.  & \\ [0.5ex]
 \hline
 \rule{0in}{2.5ex}
 {\bf Perturbed} &   Homotopy  equivalent &  Determined by the  &   Prescribed poly.\ or exp.\ upper bounds.  \\
 {\bf RAAG}     &   to $K_\Gamma$, but typically   & topology of $K_\Gamma$.  & $\exists$ $x^n$ and $e^x$ examples.  \\
                         &       not homeomorphic. &   &   Not known in general.\\ [0.5ex]
\hline
 \end{tabular}
 }
\end{center}

\bigskip

Many of the key constructions in this paper evolved from the circle of ideas started by Stallings \cite{[Stallings]} and Bieri \cite{[Bieri]}. So it is 
only fitting that  our main result, Theorem~\ref{Dehn function of the double}, shows that there exist analogues of Stallings' 
group whose Dehn functions are polynomial (of any degree at least 4) or exponential. 

\begin{mainthm}
 For each $d \in \NN \cup \{\infty\}$ there is a perturbed RAAG whose kernel subgroup $K=K(d)$ is of type $F_2$ but not $F_3$, and has 
  Dehn function given by: \[\delta_K(n)\asymp\left\{\begin{array}{rl} n^{d+3}, & d<\infty\\ e^n\,, & d=\infty\,.\end{array}\right.\]
\end{mainthm}

This paper is organized as follows. In Section~1 we recall some background facts about RAAGs.  In Section~2 we study the free-by-cyclic 
groups which will be used in the definition and construction of the perturbed RAAGs. The perturbed RAAGs are introduced in Section~3. They 
are shown to be   CAT(0) cubical groups and the topology of their kernel subgroups is investigated. Sections~4 and 5 are concerned with the 
Dehn functions of the kernels of perturbed RAAGs. An analogue of the technique of {\em pushing fillings}  \cite{[ABDDY]} is developed in 
Section~4, and is used to provide upper bounds for the Dehn function of the kernel of a general perturbed RAAG. In Section~5, we introduce 
a family of perturbed RAAGs whose kernels have lower bound estimates on their Dehn functions which agree with the upper bounds. These 
groups are perturbed versions of the RAAGs introduced in Section~2.5.2 of \cite{[Brady-Blue-Book]} and generalized in Section~5 of \cite{[ABDDY]}. 
Section~6 contains the statement and proof of the Main Theorem.

\section{Preliminaries on RAAGs} \label{back}

Given a graph $\Gamma=(V\Gamma,E\Gamma)$, recall that the right-angled Artin group $A_\Gamma$ has the presentation \[A_\Gamma\simeq\pres{V\Gamma}{[u,v]=1,\;\{u,v\}\in E\Gamma}\]

A spherical flag complex $K_\Gamma$ with $K_\Gamma^{(1)}=\Gamma$ may be constructed as follows: in the Hilbert space $H_\Gamma=\ell_2(V\Gamma)$, let a subset $S$ of $V\Gamma$ span a simplex of $K_\Gamma$ on the unit sphere iff $S$ is a clique in $\Gamma$.\\

For every $S\subset V\Gamma$, let $C(S)$ denote the (unit) cube spanned by the vectors of $S$, and let $C_\Gamma$ denote the union of all $C(S)$ such that $S$ is a clique in $\Gamma$. We may then proceed to construct a non-positively curved cubical complex $X_\Gamma$ by setting $X_\Gamma$ to be the quotient of the set $C_\Gamma$ by the action of the translation group generated by the maps $x\mapsto x+v$, $v\in V\Gamma$.\\

Recall the spherical complex construction (see section 2.1.3 in \cite{[Brady-Blue-Book]}). If $K$ is a simplicial complex, then $S(K)$ is the simplicial complex with vertex set $\{v^\pm\}_{v\in K^{(0)}}$, such that for every simplex $\sigma\in K$ and any partition $\sigma^{(0)}=A\sqcup B$, $S(K)$ will include the simplex with vertices $A^+\sqcup B^-$. 
Thus for every $P\subseteq K^{(0)}$ one has a simplicial embedding $\iota_P:K\to S(K)$ defined on $K^{(0)}$ by $\iota_P(v)=v^+$ for $v\in P$ and $\iota_P(v)=v^-$ otherwise. Clearly there is also a surjective simplicial map $\pi:S(K)\to K$ given by $\pi(v^\pm)=v$ satisfying $\pi\circ\iota_P=\mathrm{id}_K$, while the map $fold_{P}=\iota_P\circ\pi:S(K)\to S(K)$ is a retraction of $S(K)$ onto the embedded copy $K_P=\iota_P(K)$ of $K$ in $S(K)$.

The following proposition summarizes some well-known properties of the sphere construction we will be using implicitly throughout this work:
\begin{prop}[Properties of the sphere construction]\label{props of sphere construction} Let $K$ and $L$ be simplicial complexes. Then:
\begin{enumerate}
	\item If $L$ is a full sub-complex of $K$ then $S(L)$ is a full sub-complex of $S(K)$.
	\item If $K$ is a flag complex, then $S(K)$ is a flag complex.
	\item $S(L\join K)=S(L)\join S(K)$.
	\item For any simplicial map $\lambda:L\to K$ there is a unique simplicial $S(\lambda):S(L)\to S(K)$ satisfying $\lambda\circ fold_{\lambda\inv P}=fold_{P}\circ S(\lambda)$ for all $P\subset K^{(0)}$.
\end{enumerate}
\end{prop}

A well-known and immediate application of the proposition (part 2) is to the geometry of the complex $X_\Gamma$ described above (see, e.g. section 2.5.1 in \cite{[Brady-Blue-Book]}): $X_\Gamma$ has a single vertex -- denoted by $\ast$, -- whose link is the spherical complex $S(K_\Gamma)$ constructed over $K_\Gamma$. $X_\Gamma$ has the structure of a non-positively curved piecewise-Euclidean cube complex, because $S(K_\Gamma)$ is flag. Thus, $X_\Gamma$ is a $K(A_\Gamma,1)$ since its universal cover is contractible.

\section{The perturbing groups}\label{section:the groups S_d}
\subsection{Preliminary: LOG presentations and their Morse functions}\label{preliminary:LOG presentations}
Recall that a LOG presentation consists of a directed graph $\Psi$ and a labeling $\lambda:E\Psi\to V\Psi$. We use the symbols $\init e$ and $\term e$ to denote the initial and terminal vertices of an edge $e\in E\Psi$.

\begin{figure}[ht]
    \includegraphics[width=.5\textwidth]{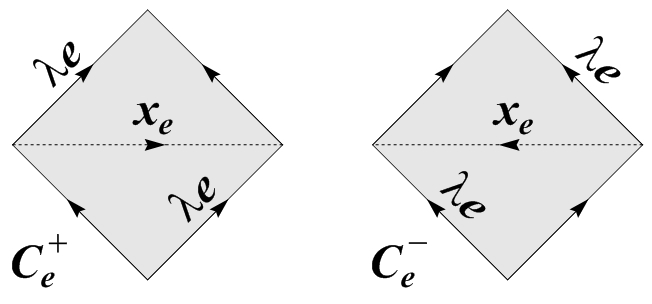}
    \caption{\tiny Two orientations of relator squares: $\cell{e}^+$ (left) and $\cell{e}^-$ (right). Note how the Morse function $h$ sorts the edges of $\cell{e}^\pm$ into two pairs: the upper edges and the lower edges.\normalsize\label{orientations of squares}}
\end{figure}

From a LOG presentation $(\Psi,\lambda)$ we construct a group $G(\Psi,\lambda)$ with generating set $V\Psi$ and relations of the form \[r_e=(\lambda e)(\term e)(\lambda e)\inv(\init e)\inv\,\]
-- see figure \ref{orientations of squares}.\\

The relations $r_e$ imply there is a unique epimorphism $h:G(\Psi,\lambda)\to\ZZ$ defined by $h(V\Psi)=\{1\}$. Moreover, let $\Cay{\Psi,\lambda}$ denote the Cayley $2$-complex of the corresponding group presentation. Then $h$ extends linearly to a real-valued Morse function of $\Cay{\Psi,\lambda}$. The corresponding presentation $2$-complex will be denoted by $P(\Psi,\lambda)$, its only vertex denoted henceforth by $(\ast)$. The function $h$ is then the lift of a circle-valued Morse function $\bar{h}:P(\Psi,\lambda)\to\fat{S}^1$.\\

We mark the homotopy classes of certain loops in $P(\Psi,\lambda)$: For each edge $e\in E\Psi$, let $\cell{e}$ denote the unoriented closed $2$-cell (a Euclidean unit square) corresponding to the relation $r_e$, and let $x_e\in G(\Psi,\lambda)$ correspond to the homotopy class of the loop arising as the intersection of $\bar{h}\inv(1\in\fat{S}^1)$ with the image of $\cell{e}$ in $P(\Psi,\lambda)$, oriented so that $\init e\cdot x_e=\lambda e$ holds in $G(\Psi,\lambda)$ (and consequently $x_e\term e=\lambda e$). We subdivide $\cell{e}$ into a pair of triangles by appropriately marking and orienting the diagonal corresponding to $x_e$, and it is possible to consider two orientations of $\cell{e}$: $\cell{e}^+$ will denote $\cell{e}$ with the orientation whose boundary reads $(\lambda e)(\term e)(\lambda e)\inv(\init e)\inv$, and $\cell{e}^-$ will denote $\cell{e}$ taken with the reverse orientation. The {\it rim of $\cell{e}^+$} is defined to be the word $\vrim{\cell{e}^+}=(\init e)\inv(\lambda e)$, while $\vrim{\cell{e}^-}=(\lambda e)\inv(\init e)$ (see fig. \ref{orientations of squares}). Since in $G(\Psi,\lambda)$ one has $\vrim{\cell{e}^{\pm}}=x_e^{\pm 1}\in\ker(h)$, we will use $\erim{\cell{e}^\pm}$ to denote the word $x_e^{\pm 1}$.\\

The complex $P(\Psi,\lambda)$ is completely determined by $\link{\ast}{P(\Psi,\lambda)}$, and recall $\Cay{\Psi,\lambda}$ is CAT(0) if and only if $\link{\ast}{P(\Psi,\lambda)}$ has girth at least $4$. Also, $\Cay{\Psi,\lambda}$ is Gromov-hyperbolic whenever the girth of $\link{\ast}{P(\Psi,\lambda)}$ is at least 5: the group $G(\Psi,\lambda)$ acts geometrically on the CAT(0) space $\Cay{\Psi,\lambda}$ while the latter contains no $2$-flats.\\

Let $V^+$ (resp. $V^-$) denote the set of symbols $v^+$ (resp. $v^-$) with $v\in V\Psi$. Let $V^\pm=V^+\cup V^-$. Then $\link{\ast}{P(\Psi,\lambda)}$ has the following description: it is the graph with vertex set $V^{\pm}$ and two vertices $u^{\epsilon_1},\,v^{\epsilon_2},\,\epsilon_{1,2}\in\{+,-\}$ are joined by an edge if and only if there is and edge $e\in E\Psi$ such that a pair of adjacent edges $f,g$ of $\cell{e}$ are labeled $u$ and $v$ respectively, and $\partial_{\epsilon_1}f=\partial_{\epsilon_2}g$. Recall that (with respect to the Morse function $h$) the subgraph $\dnlink{\ast}{P(\Psi,\lambda)}$ spanned by $V^+$ is called {\it the descending link} of $\ast$, while the full subgraph $\uplink{\ast}{P(\Psi,\lambda)}$ induced by $V^-$ is the {\it ascending link} of $\ast$.\\

\subsection{Non-positively curved LOG presentations with nice links}
In this paper we only deal with situations where $\link{\ast}{P(\Psi,\lambda)}$ has girth at least $4$, so that $\link{\ast}{P(\Psi,\lambda)}$ is always a simple graph. Moreover, an important property of the presentations we are going to deal with is that both $\dnlink{\ast}{P(\Psi,\lambda)}$ and $\uplink{\ast}{P(\Psi,\lambda)}$ are trees. We assume these properties of $\link{\ast}{P(\Psi,\lambda)}$ until the end of the section.\\

Morse theory shows $\ker(h)$ is a free group generated by the set $\{x_e\}_{e\in E\Psi}$. Thus, each level set of $h$ in $\Cay{\Psi,\lambda}$ is a tree, with each edge naturally labeled by some $x_e$.\\

By an {\it ascending word of height $n$} we shall mean a positive word $w=\sigma_1\cdots\sigma_n$ in the generating set $V\Psi$. Observe that $h(\sigma_1\cdots\sigma_k)=k$ for all $0\leq k\leq n$. Since $\link{\ast}{P(\Psi,\lambda)}$ has girth at least $4$, every $C_e$ ($e\in E\Psi)$ is embedded in $\Cay{\Psi,\lambda}$ isometrically (due to non-positive curvature). Together these imply that any lift of an ascending word to $\Cay{\Psi,\lambda}$ is a geodesic in the standard piecewise-Euclidean metric on $\Cay{\Psi,\lambda}$.\\

Given a pair $(u,v)$ of ascending words, consider the element $x=uv\inv\in\ker(h)$. Since $x$ has a unique representative as a word $w$ in the alphabet $\{x_e\}_{e\in E\Psi}$, one expects to be able to produce a standardized embedded (in $\Cay{\Psi,\lambda}$) filling of the word $uv\inv w\inv$.\\

%In this paragraph we show a procedure for constructing such embedded fillings, resulting in a measure of control sufficient for computing the distortion of $\ker(h)$ in the groups $G(\Psi,\lambda)$ that we plan to consider.\\

Let $(a^+,b^+)$ be an oriented edge in $\dnlink{\ast}{P(\Psi,\lambda)}$. Then there is a unique edge $e\in E\Psi$ such that either $\lambda e=a$ and $\term e=b$ or $\lambda e=b$ and $\term e=a$. Let $\cell{(a,b)}$ denote the cell $\cell{e}$ taken with a matching orientation: choose the orientation on $\cell{e}$ so that $\vrim{\cell{e}}ba\inv$ is the word read along the boundary of $\cell{e}$.\\ %, always of the form $u\inv v$ for suitable $u,v\in V$.\\

\begin{figure}[ht]
    \includegraphics[width=.6\textwidth]{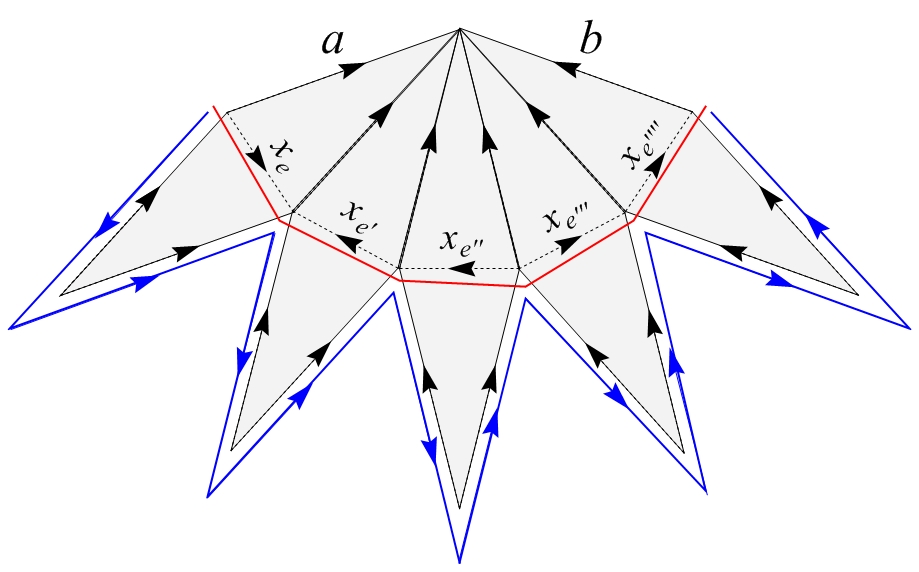}
    \caption{\tiny A typical fan $\mathrm{Fan}(a,b)$ connecting a pair of ascending edges. The vertex rim $\vrim{a,b}$ is marked in blue; the edge rim $\erim{a,b}$ is marked in red.\normalsize\label{fig:simple fan}}
\end{figure}

More generally, let $(a^+=v^+_0,\ldots,v^+_k=b^+)$ be a simple vertex path (that is, no vertex is repeated) in $\dnlink{\ast}{P(\Psi,\lambda)}$. Since $\dnlink{\ast}{P(\Psi,\lambda)}$ is a tree, this path is completely determined by the choice of the letters $a,b\in V\Psi$. Let $\mathrm{Fan}(a,b)$ denote the quotient of the disjoint union of the closed cells $\{C_i=\cell{(v_{i-1},v_i)}\}_{i=1}^k$ under identifying the upper edges of $\cell{(v_{i-1},v_i)}$ and $\cell{(v_i,v_{i+1})}$ which are marked with $v_i$, for all $i=1,\ldots,k$ (This is well-defined, because $\link{\ast}{P(\Psi,\lambda)}$ contains no edge-loops). We define the {\it vertex-} and {\it edge-rims} of the fan $\mathrm{Fan}(a,b)$ as the products of the respective rims of its constituent cells (see figure \ref{fig:simple fan}):
\begin{eqnarray*}
	\vrim{a,b}&=&\vrim{\cell{(v_0,v_1)}}\cdots\vrim{\cell{(v_{k-1},v_k)}}\,,\\
	\erim{a,b}&=&\erim{\cell{(v_0,v_1)}}\cdots\erim{\cell{(v_{k-1},v_k)}}\,.
\end{eqnarray*}
We now claim that for any $a,b,u\in V\Psi$ with $a\neq b$, the word $\vrim{a,b}$ does not contain $uu\inv$ as a sub-word. Suppose not so, and let $(a^+=v^+_0,\ldots,v^+_k=b^+)$ be a simple vertex path in $\dnlink{\ast}{P(\Psi,\lambda)}$. Then for some $0<i<k$ we have that $uv_i$ is a segment of the boundary of the cell $C_i$ while $v\inv_i u\inv$ is a segment of the boundary of $C_{i+1}$. Since $\link{\ast}{P(\Psi,\lambda)}$ contains no circuits of length $1$ or $2$, we conclude that $C_{i+1}$ must equal $C_i$ with an opposite orientation. This implies $v_{i-1}=v_{i+1}$ -- a contradiction. Thus, for any fan $\mathrm{Fan}(a,b)$, the only possible elementary cancellations in $\vrim{a,b}$ are of the form $u\inv u$, $u\in V\Psi$. By construction, such a substring only occurs on the rim of one of the cells $C_i$ (same notation as above). But this is impossible, as $\link{\ast}{P(\Psi,\lambda)}$ contains no circuits of length $1$. We conclude:

\begin{lemma} For any pair of distinct letters $a,b\in V\Psi$, the vertex rim of $\mathrm{Fan}(a,b)$ is a reduced word in the letters of $V\Psi$, and the edge-rim of $\mathrm{Fan}(a,b)$ is a reduced word in the letters $\{x_e\}_{e\in E\Psi}$.\hfill $\square$
\end{lemma}

We would now like to extend the construction of fans to fans `joining' a pair of ascending words of equal height (figure \ref{fig:links_in_fans} illustrates the general idea). Let $u,v$ be positive words of length $n$. A $(u,v)$-fan of height $n$ is defined by induction on $n$ as follows:
\begin{itemize}
	\item[-] For $n=1$, we have $u,v\in V\Psi$, in which case either $u\neq v$ and $\mathrm{Fan}(u,v)$ is exactly as we have just defined above, or $u=v$ and then $\mathrm{Fan}(u,v)$ is defined to consist of one directed edge labeled by the symbol $u$ (same as $v$).
	\item[-] For $n>1$, write $u=au'$ and $v=bv'$ with $a,b\in V\Psi$. A $(u,v)$-fan consists of a $(u',v')$-fan $F'$ with rim $a_1\inv b_1\cdots a_k\inv b_k$, and with additional $1$-fans attached along the rim: let $b_0=a$ and $a_{k+1}=b$; we attach the fans $F_i=\mathrm{Fan}(b_{i-1},a_i)$ for $i=1,\ldots,k+1$, in that order, by identifying top edges on each $1$-fan to the corresponding edges along the rim of $F'$. We define $F$ to be the result of these identifications. Note the possibility that $b_{i-1}=a_i$ for some $i$: this case implies a cycle of length $3$ or a cycle of length $1$ in $\link{\ast}{P(\Psi,\lambda)}$ -- which is impossible. Thus, the word obtained from concatenating the rims of the fans $F_i$ is reduced, and we set it to be the rim of $F$.
\end{itemize}
Observe that a $(u,v)$-fan $F$ is uniquely determined by the pair $(u,v)$, so henceforth denote $F=\mathrm{Fan}(u,v)$. Also, $\vrim{F}\in\ker(h)$ can be rewritten uniquely into a word of length $\frac{1}{2}|\vrim{F}|$ in the alphabet $\{x_e\}_{e\in E\Psi}$: we denote this word by $\erim{F}$, and note that it is a geodesic in the Cayley graph of $\ker(h)$ (with respect to the letters $\{x_e\}_{e\in E\Psi}$, this graph is a tree). Figure \ref{fig:links_in_fans} illustrates the observations made thus far.\\

\begin{figure}[ht]
    \includegraphics[width=.6\textwidth]{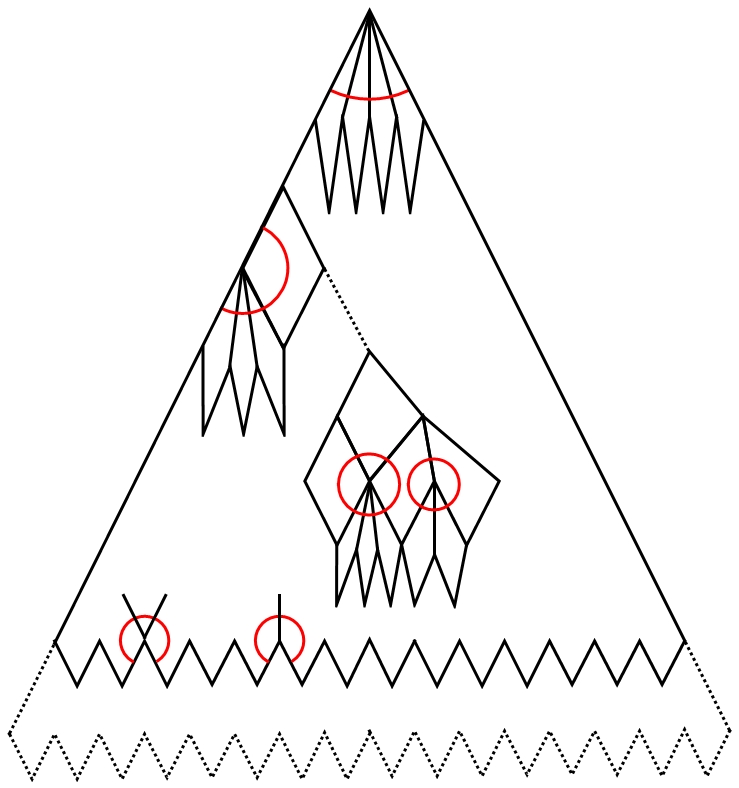}
    \caption{\tiny The different types of vertex links (red) in $\mathrm{Fan}(u^n,v^n)$.\normalsize\label{fig:links_in_fans}}
\end{figure}

We may treat $\mathrm{Fan}(u,v)$ as a disk diagram for $G(\Psi,\lambda)$. As a result, for any vertex $x\in\Cay{\Psi,\lambda}$ this defines a natural cellular map $\mathrm{Fan}(u,v)\to\Cay{\Psi,\lambda}$ sending the vertex of the fan to $x$. Observe that the level sets of $\Cay{\Psi,\lambda}$ are trees. Since edge-rims of fans are reduced, the intersection of $\mathrm{Fan}(u,v)$ with every level set is an arc (corresponding to the rim of a sub-fan of appropriate height). This implies that no two vertices of equal heights in $\mathrm{Fan}(u,v)$ get identified under this map. By construction of $\mathrm{Fan}(u,v)$, no two vertices of different heights get identified. We conclude:
\begin{prop}\label{fans are embedded} Suppose $\Psi$ is a graph with labeling $\lambda:E\Psi\to V\Psi$. If $\link{\ast}{P(\Psi,\lambda)}$ is triangle-free and $\uplink{\ast}{P(\Psi,\lambda)},\dnlink{\ast}{P(\Psi,\lambda)}$ are trees, then for any pair $(u,v)$ of ascending words in the letters of $V\Psi$, the fan $\mathrm{Fan}(u,v)$ embeds in $\Cay{\Psi,\lambda}$.
\end{prop}
This embedding property of fans will allow us to construct diagrams of prescribed area in some modified right-angled Artin groups. The following trivial observation becomes meaningful in this context:
\begin{remark}\label{remark:exp upper bound} Suppose $\Psi$ is a graph with labeling $\lambda:E\Psi\to V\Psi$. If $\link{\ast}{P(\Psi,\lambda)}$ is triangle-free and $\uplink{\ast}{P(\Psi,\lambda)},\dnlink{\ast}{P(\Psi,\lambda)}$ are trees. Then there is a constant $C$ such that for any $n\in\NN$ and any pair $(u,v)$ of ascending words of length $n$ in the letters of $V\Psi$ one has
\[\left|\erim{\mathrm{Fan}(u,v)}\right|\leq C^n\,,\quad \area{\mathrm{Fan}(u,v)}\leq C^{n+1}\,.\]
Clearly, $C$ is the maximum length of a simple fan (fan of height one).
\end{remark}

\subsection{Polynomial distortion -- the group $S_d$}\label{description:poly}
We shall now study the group \[S_d=S(a_0,\ldots,a_d,s)=G(\Psi_d,\lambda_d)\] obtained from the LOG presentation $(\Psi_d,\lambda_d)$ given in figure \ref{figure:LOG_poly} for $d\geq 0$. One has:
\begin{displaymath}
	S_d=\left\langle
		a_0,\ldots,a_d,s \; \left|
				\left[s,a_0\right]=1\,,\;
				a_{i+1}a_ia\inv_{i+1}=a_0\,,\;i<d
		\right.\right\rangle
\end{displaymath}
We henceforth suppress any mention of $\lambda_d$, and denote $\PP{d}=P(\Psi_d)$ and $\tilPP{d}$ for $\Cay{\Psi_d}$.\\

\begin{figure}[ht]
    \includegraphics[width=\textwidth]{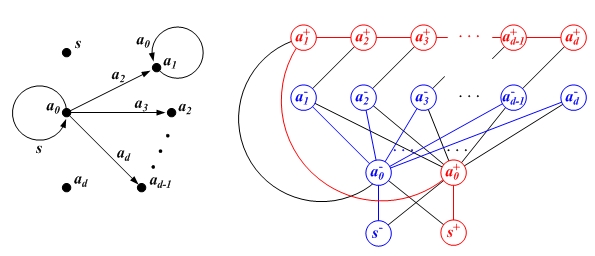}
    \caption{\tiny LOG presentation $(\Psi_d,\lambda_d)$ and vertex link $\link{\ast}{\PP{d}}$ for the group $S_d$ and $d\geq 1$.\normalsize\label{figure:LOG_poly}}
\end{figure}

It is clear that $\link{\ast}{\PP{d}}$ has girth $4$, making $\tilPP{d}$ into a $2$-dimensional CAT(0) cube complex. The ascending and descending links of the marked vertex are both trees, so that fans are well-defined and embedded in $\tilPP{d}$.\\

We want to study the family of fans $\mathrm{Fan}(u^n,v^n)$, where $n\in\NN$ and $u,v$ are distinct letters in $V\Psi_d$. We want to calculate the length of $\erim{\mathrm{Fan}(u^n,v^n)}$.\\

More specifically, consider fans of the form $F_{i,j}(n)=\mathrm{Fan}(a_i^n,a_j^n)$ and $F_{0,j}(n)=\mathrm{Fan}(a_0^n,a_j^n)$ for $1\leq i<j\leq d$. We denote \[f_{i,j}(n)=\frac{1}{2}\left|\vrim{F_{i,j}(n)}\right|\,,\]
and set $f_{i,i}(n)$ to equal zero for all $i$ and $n$ as well as $f_{i,j}(0)=0$ for all $i\leq j$. From the structure of $\dnlink{\ast}{\PP{d}}$ we have:
\[f_{i,j}(1)=j-i\,,\]
and then it is also possible to observe the following recursive relations:
\begin{lemma}\label{polynomial recursion} The following holds for $n\geq 1$ and $0\leq i<j\leq d$:
\[(\ast)\quad f_{i,j}(n)=f_{i,j-1}(n)+f_{0,j}(n-1)+1\,.\]
\end{lemma}

\begin{figure}[ht]
    \includegraphics[width=\textwidth]{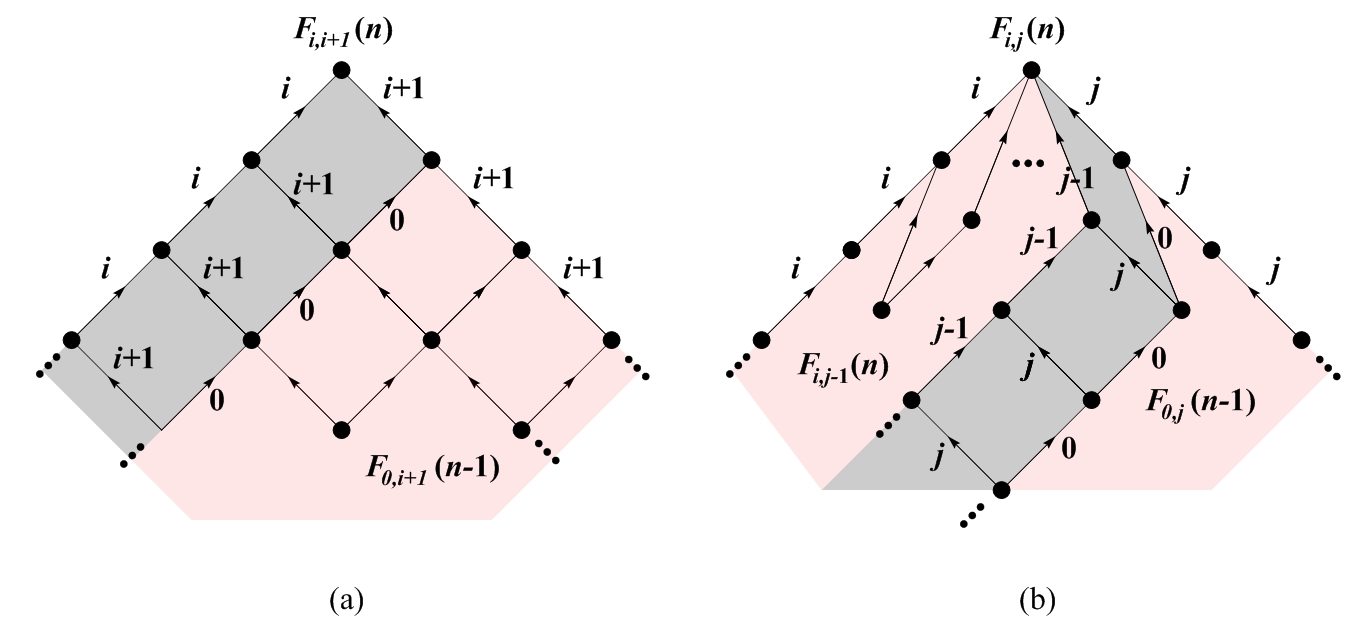}
    \caption{\tiny Decomposing fans using galleries (the indices $0,i,j$ etc. are used instead of the symbols $a_0,a_i,a_j$ respectively.)\normalsize\label{fig:fan_decomposition}}
\end{figure}

\begin{proof} First suppose $i+1=j$. Figure \ref{fig:fan_decomposition}(a) explains how the fan $F_{i,i+1}(n)$ is the union of a vertical gallery of identical cells with a fan of type $F_{0,i+1}(n-1)$, resulting in
\begin{displaymath}
	(\ast)'\quad f_{i,i+1}(n)\;=\;1+f_{0,i+1}(n-1)\,,\quad n\geq 1.
\end{displaymath}
Next, considering the case when $i<j-1$, figure \ref{fig:fan_decomposition}(b) explains how $F_{i,j}(n)$ decomposes into the union of $F_{i,j-1}(n)$, a gallery, and $F_{0,j}(n-1)$, resulting in the recursion $(\ast)$ for this case ($n\geq 2$).

We are only left to observe that the boundary conditions
\[f_{i,j}(0)=0\,,\quad f_{i,i}(n)=0\]
are consistent with the recursion: note how $(\ast)'$ becomes a special case of $(\ast)$ when the boundary conditions are invoked in the case $j=i+1$, and observe that $(\ast)$ holds for $n=1$ as well:
\[1+f_{i,j-1}(1)+f_{0,j}(0)=1+(j-1-i)+0=j-i=f_{i,j}(1)\,.\]
This completes the proof.
\end{proof}
We are now ready to prove:
\begin{prop}\label{polynomial growth of rims} For all $0\leq i<j\leq d$, the function $f_{i,j}$ is a polynomial in $n$ of degree $j$, in fact:
\begin{displaymath}
	f_{i,j}(n)=\sum_{k=i+1}^j\binom{n+k-1}{n-1}
\end{displaymath}
\end{prop}
\begin{proof} First set $Q_j(n)=\binom{n+j}{n}-1$. One verifies that $Q_j(n)$ satisfies the recursive identity $(\ast)$ with $i=0$, and the same boundary conditions as $f_{0,j}(n)$. Therefore,
\begin{displaymath}
	f_{0,j}(n)=\binom{n+j}{n}-1\,.
\end{displaymath}
Applying $(\ast)$ an additional $(j-i)$ times, one has
\begin{displaymath}
	f_{i,j}(n)
		=(j-i)+\sum_{k=i+1}^j f_{0,k}(n-1)
		=(j-i)+\sum_{k=i+1}^j\left[\binom{n+k-1}{n-1}-1\right]\,,
\end{displaymath}
which produces the required result.
\end{proof}

We have not yet considered the most useful fans $\mathrm{Fan}(s^n,a_j^n)$, $j\leq d$. For $d\geq 1$ ($\Psi_d$ is not defined for $d<2$ anyway), looking at $\dnlink{\ast}{\PP{d}}$ it is obvious that $\mathrm{Fan}(s^n,a_0^n)$ and $\mathrm{Fan}(a_0^n,a_j^n)$ ($j\leq d$) may be concatenated to form a fan from $s^n$ to $a_j^n$; by uniqueness of fans, the result of this concatenation must be $\mathrm{Fan}(s^n,a_j^n)$. Since the letters $s$ and $a_0$ commute, $\erim{\mathrm{Fan}(s^n,a_0^n)}=n$, and we conclude:
\begin{prop}\label{cor:polynomial growth of rims} For all $a_d\neq u\in V\Psi_d$, we have: \[\left|\erim{\mathrm{Fan}(u^n,a_d^n)}\right|\asymp n^d\,,\quad \area{\mathrm{Fan}(u^n,a_d^n)}\asymp n^{d+1}\,.\quad\square\]
\end{prop}

\bigskip

\subsubsection*{Ascending Fans}\label{ascending fans} A completely symmetric construction to the construction of fans above is that of {\it ascending fans}: use the analogous definitions and constructions for a pair of {\it descending} words $u,v$ of equal height (words written in the inverses of the positive generators only). Since $\uplink{\ast}{\PP{d}}$ is a tree as well, all the preceding general results apply to ascending fans as well.

One needs to study the growth of rims for ascending fans of the form $\mathrm{Fan}\left(u^{-n},v^{-n}\right)$ ($u,v\in V\Psi_d$, $u\neq v$) as well. Again, looking at the diagram of $\link{\ast}{\PP{d}}$, one sees that ascending fans of this kind are embedded in the Cayley $2$-complex of the LOG presentation. We are therefore left to study only the growth of $\erim{\mathrm{Fan}\left(u^{-n},v^{-n}\right)}$.

For the purpose of the computation, we shall (again) denote by $F_{i,k}(n)$ the (ascending) fan with $u=a_i$ and $v=a_k$ and $F_{s,i}$ for $v=a_i$ and $u=s$ respectively. We use lower case $f$ instead of the upper case $F$ in this notation to denote the length of the horizontal rim of the corresponding fans.

\begin{figure}[h]
    \includegraphics[width=\textwidth]{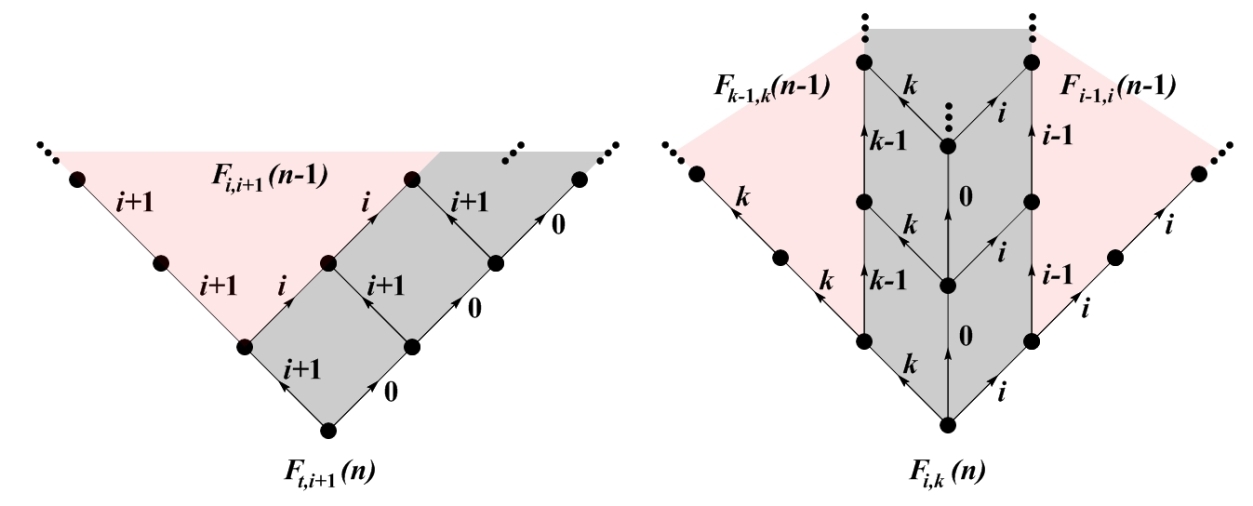}
    \caption{\tiny Decomposing ascending fans (the indices $i,k$ are used instead of the symbols $a_i,a_k$ with $0\leq i,k\leq d$).\normalsize\label{fig:ascending_fan_decomposition}}
\end{figure}

The figures above demonstrate the following recursive relations:
\begin{enumerate}
	\item $f_{0,i+1}(n)=1+f_{i,i+1}(n-1)$
	\item $f_{i,k}(n)=2+f_{i-1,i}(n-1)+f_{k-1,k}(n-1)$
\end{enumerate}
Since $a_0^-$ separates $s^-$ in $\uplink{\ast}{\PP{d}}$, and $[s,a_0]=1$, we will have that
\begin{displaymath}
	f_{s,0}(n)=n\,,\quad f_{s,i}(n)=f_{s,0}(n)+f_{0,i}(n)=n+f_{0,i}(n)
\end{displaymath}
Since $a_1$ and $a_0$ commute, we will also have
\begin{displaymath}
	f_{0,1}(n)=n.
\end{displaymath}
In addition, an obvious base value is $f_{i,k}(1)=2$ for $i,k\geq 1$.\\

In a manner similar to the case of descending fans, for $1\leq i<k$ one shows that $f_{i,k}(n)$ is a polynomial in $n$ of degree at most $k$. Since $s$ commutes with $a_0$ we deduce that $f_{s,k}(n)$ and $f_{0,k}(n)$ are, too, polynomials of degree at most $k$ in $n$, and we may conclude:
\begin{prop}\label{lemma:ascending fans are polynomial} For all $u,v\in V\Psi_d$, we have: \[\left|\erim{\mathrm{Fan}(u^{-n},v^{-n})}\right|\preceq n^d\,,\quad \area{\mathrm{Fan}(u^{-n},v^{-n})}\preceq n^{d+1}.\quad\square\]
\end{prop}

\subsection{Exponential distortion -- the group $S_{\infty}$}\label{description:exp}

Here we remark on the distortion of $\ker(h)$ in the group $S_{\infty}$ obtained from the LOG presentation in figure \ref{figure:LOG_exp}.
Observing that the girth of $\link{\ast}{\PP{\infty}}$ equals $5$, we conclude $S_{\infty}$ is word-hyperbolic: $\tilPP{\infty}$ is CAT(0) and contains no flat plane, and hence it is Gromov-hyperbolic (see \cite{[Bridson-Haefliger]}, chapter III.$\Psi$, theorem 3.1). In particular, geodesics in $\tilPP{\infty}$ diverge (at least) exponentially.\\

\begin{figure}[ht]
    \includegraphics[width=\textwidth]{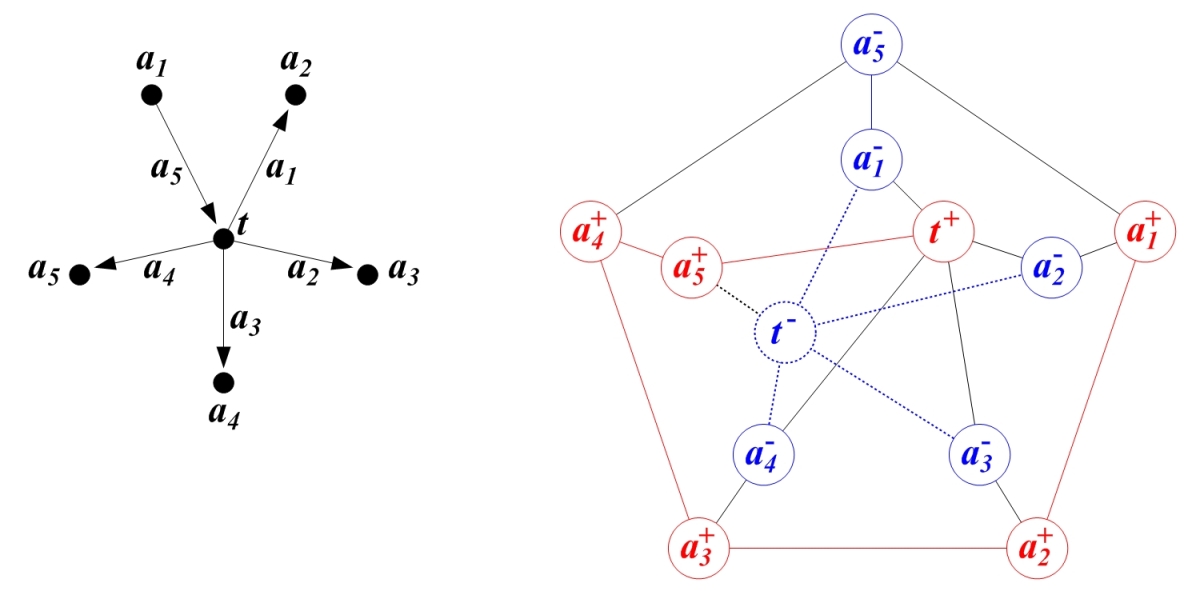}
    \caption{\tiny LOG presentation $\Psi_\infty$ and vertex link $\link{\ast}{\PP{\infty}}$ for the group $S_{\infty}$.\normalsize\label{figure:LOG_exp}}
\end{figure}

We consider the ascending words $u_n=a_1^n$ and $v_n=a_3^n$. Observe that the combinatorial distance of $a_1^+$ to $a_1^-$ in $\link{\ast}{\PP{\infty}}$ is $2$, and similarly for $a_3$. This guarantees that $(u_n)_{n=1}^\infty$ and $(v_n)_{n=1}^\infty$ are situated along geodesic rays -- denote them with $\gamma_1$ and $\gamma_3$ -- in $\tilPP{\infty}$ (with respect to the piecewise-Euclidean metric). Also, the combinatorial distance from $a_1^+$ to $a_3^+$ equals $2$, implying $\angle(\gamma_1,\gamma_3)=\pi$ and hence $\gamma_1(\infty)\neq\gamma_3(\infty)$. Hence, if $w_n$ is any word in the letters $\{x_e\}_{e\in E\Psi_{\infty}}$ satisfying $w_n=u_nv_n\inv$, then $w_n$ avoids the ball of radius $n$ about $u_n$, and hence $|w|$ is at least exponential in $n$, by exponential divergence of geodesics (in a Gromov-Hyperbolic space). The implication for (descending) fans is:
\begin{prop}\label{prop:exponential}
For the presentation of $S_\infty$, one has $\left|\erim{\mathrm{Fan}(a_1^n,a_3^n)}\right|\asymp e^n$. In particular, $\area{\mathrm{Fan}(a_1^n,a_3^n)}\asymp e^n$.
\end{prop}

\section{Perturbed RAAGs (PRAAGs)}
Throughout this section, let $\Gamma$ be a simple (no loops, no double edges) directed graph. Every $e\in E\Gamma$ has well-defined initial and terminal vertices $\init e$ and $\term e$, respectively.\\

Recall from section \ref{back} that the Artin group $A_\Gamma$ has a $K(A_\Gamma,1)$ in the form of the non-positively curved cubical complex $X_\Gamma$, as well as an epimorphism (or height function) $h$ of $A_\Gamma$ onto $\ZZ$, coinciding with the restriction (to $\widetilde{X}_\Gamma^{(0)}$) of the lift of a Morse function $\bar{h}:X_\Gamma\to\fat{S}^1$ and sending every $v\in V\Gamma$ to $1\in\ZZ$.\\

The same constructions may be applied to any sub-graph $H\subset\Gamma$, inducing convex embeddings $X_H\hookrightarrow X_\Gamma$ and consequently convex isometric embeddings $\widetilde{X}_H\hookrightarrow\widetilde{X}_\Gamma$ and monomorphisms $A_H\hookrightarrow A_\Gamma$ commuting with the respective Morse/height functions.

\subsection{Definition of PRAAGs}
To simplify notation let $\NN$ denote the set of natural numbers together with $0$, and let $\NNhat$ denote $\NN\cup\{\infty\}$.\\

For every edge $f\in E\Gamma$ consider the link $\link{f}{K_\Gamma}$ of $f$ in the flag complex $K_\Gamma$ generated by $\Gamma$: to be sure, $\link{f}{K_\Gamma}$ consists of all (open) simplices $\sigma\in K_\Gamma$ such that the join $\sigma\join f\in K_\Gamma$.\\

A {\it marking} on $\Gamma$ is a function $d:E\Gamma\to\NNhat$. We say that a marking $d$ on $\Gamma$ is {\it admissible}, if the set $F$ of edges $f\in E\Gamma$ with $d(f)\geq 1$ satisfies the following condition: for any $e,f\in F$, either $e=f$ or $\bar e\cap\link{f}{K_\Gamma}=\varnothing$.\\

We fix an admissible marking $d$ till the end of this section.\\

For $f\in E\Gamma$, let $\Sigma_f$ denote the sub-complex obtained by deleting the open star of $f$ from its closed star in $K_\Gamma$. We have:
\begin{displaymath}
	\Sigma_f=\link{f}{K_\Gamma}\join\{\init f,\term f\}\,.
\end{displaymath}
The marking $d$ is admissible, so observe that $\Sigma_f$ is a full sub-complex of $K_{\Gamma-F}$ for every $f\in F$. Corresponding to this in $\link{\ast}{X_{\Gamma-F}}=S(K_{\Gamma-F})$ is the full sub-complex
\begin{displaymath}
	S(\Sigma_f)=S(\link{f}{K_\Gamma})\join S\left(\{\init f,\term f\}\right)\,.
\end{displaymath}
Let $\Gamma_f$ denote the $1$-skeleton of $\Sigma_f$ and let $\Lambda_f\subset\Gamma_f$ denote the $1$-skeleton of $\link{f}{K_\Gamma}$ (see figure \ref{fig:lambda}).

\begin{figure}[hbt]
    \includegraphics[width=0.3\textwidth]{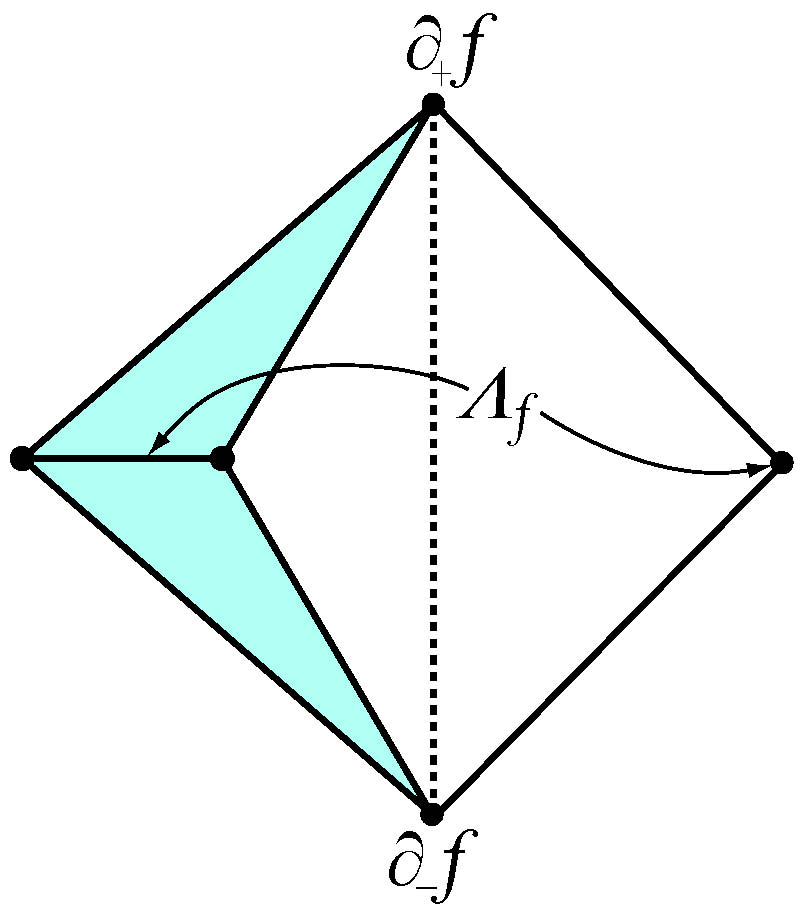}
    \caption{\tiny $\Sigma_f$ and $\Lambda_f$ for an edge $f\in E\Gamma$.\normalsize\label{fig:lambda}}
\end{figure}

We form the space $Y_f=X_{\Lambda_f} \times \PP{d(f)}$. We denote the vertices of $\Psi_{d(f)}$ by $s_f$ and $a_{f,i}$ in correspondence with the generators $s$ and $a_i$ discussed in section \ref{section:the groups S_d} for the groups $S_d$, $i=0,\ldots,d(f)$ for $d(f)\in\NN$ and $i=1,\ldots,5$ for $d(f)=\infty$.\\

The NPC cubical complex $Y_f$ is completely determined by the link of its only vertex:
\begin{displaymath}
	\link{\ast}{Y_f}=S\left(\link{f}{K_\Gamma}\right)\join\link{\ast}{\PP{d(f)}}\,.
\end{displaymath}
For $1\leq d(f)<\infty$ observe that the four vertices $s_f^\pm$ and $a_{f,d(f)}^\pm$ in $\link{\ast}{\PP{d(f)}}$ have no edges joining them (see figure \ref{figure:LOG_poly}). For $d(f)=\infty$, the quadruple of vertices $a_{f,1}^\pm$ and $a_{f,3}^\pm$ in $\link{\ast}{\PP{\infty}}$ has the same property (figure \ref{figure:LOG_exp}). We define a gluing map
\begin{displaymath}
\varphi_f:S(\Sigma_f) \to \link{\ast}{Y_f}= S\left(\link{f}{K_\Gamma}\right) \join \link{\ast}{\PP{d(f)}}
\end{displaymath}
by setting $\varphi_f$ to be the identity on $\link{f}{K_\Gamma}$ (and hence on $S(\link{f}{K_\Gamma})$) and setting $\varphi_f(\init f^\pm)=s_f^\pm$ and $\varphi_f(\term f^\pm)=a_{f,d(f)}^\pm$ when $d(f)\neq \infty$, or $\varphi_f(\init f^\pm)=a_{f,1}^\pm$ and $\varphi_f(\term f^\pm)=a_{f,3}^\pm$ for $d(f)=\infty$. The partial map $\varphi_f$ thus defined then extends uniquely to a simplicial embedding of the whole of $S(\Sigma_f)$ in $\link{\ast}{Y_f}$. The resulting map $\varphi_f$ then induces an embedding
\begin{displaymath}
	\Phi_f:X_{\Gamma_f}\to Y_f=X_{\Lambda_f}\times\PP{d(f)}\,.
\end{displaymath}

Finally, define the group $A_\Gamma(F,d)$ to be the fundamental group of the cubical complex $X_\Gamma(F,d)$ resulting from attaching the spaces $Y_f$ to $X_{\Gamma-F}$ along the maps $\{\Phi_f\}_{f\in F}$. The admissibility of the marking ensures that $X_\Gamma(F,d)$ is well-defined. Thus, $A_\Gamma(F,d)$ can be expressed as the fundamental group of a tree of groups (an amalgam) with vertex groups $A_{\Gamma-F}$ and $A_{\Lambda_f}\times S_{d(f)}$ for each $f\in F=\{f_1,\ldots,f_n\}$, amalgamated over subgroups of the form $A_{\Lambda_f}\times F_2$ -- see figure \ref{fig:tree of groups}.
%STOPPED HERE
\begin{figure}[ht]
\begin{tikzpicture}[description/.style={fill=white,inner sep=4pt}]
	\matrix(m)[matrix of math nodes, row sep=2em, column sep=3em, text height=1.5ex, text depth=0.25ex, font=\large]{
			&	 	& {A_{\Lambda_{f_1}}\times S_{d(f_1)}} \\
			&	 	& {A_{\Lambda_{f_2}}\times S_{d(f_2)}} \\
		{A_{\Gamma-F}} &  &  \\
			& 	&	{A_{\Lambda_{f_n}}\times S_{d(f_n)}} \\
	};
	\path[solid,font=\normalsize]
	(m-3-1) edge node[description,sloped] {$A_{\Lambda_{f_1}}\times F_2$} (m-1-3.base west)
					edge node[description,sloped] {$A_{\Lambda_{f_2}}\times F_2$} (m-2-3.west)
					edge node[description,sloped] {$A_{\Lambda_{f_n}}\times F_2$} (m-4-3.west);
	\path[loosely dotted] (m-2-3.south) edge (m-4-3.north);
\end{tikzpicture}
	\caption{\tiny The group $A_\Gamma(F,d)$ as a tree of groups.\label{fig:tree of groups}}
	\normalsize
\end{figure}
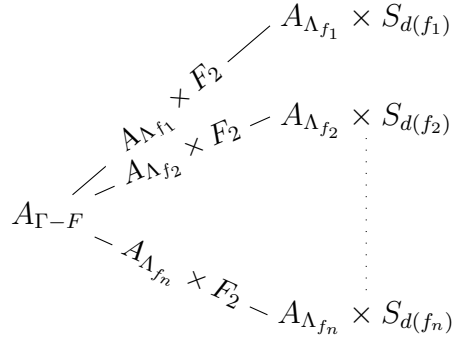
\begin{remark}[Degree zero edges] Carrying out the plumbing-and-filling on an individual degree zero (edges $f\in E\Gamma$ with $d(f)=0$) results in no alteration of $X_\Gamma$.
\end{remark}
\begin{remark}[Degree one edges] Carrying out the prescribed gluings for degree one edges ($d(f)=1$) is equivalent to subdividing the edge $f$ in $K_\Gamma$ once and adding the corresponding Artin relations to the presentation of $A_\Gamma$.
\end{remark}
In the next paragraph, we shall prove that $X_\Gamma(F,d)$ is, in fact, non-positively curved, and that the transition from $X_\Gamma$ to $X_\Gamma(F,d)$ does not alter the homotopy type of the ascending and descending links of the vertex $(\ast)$.

\subsection{Curvature conditions}
Recall that a cubical complex is non-positively curved (NPC) if and only if the link of each vertex is a flag complex. Suppose we are given a pair of NPC finite-dimensional cubical complexes $\hsm{C}_1$ and $\hsm{C}_2$, with $\hsm{C}_i^{(0)}$ containing only one vertex $v_i$. In each $\hsm{C}_i$ we identify a sub-complex $Q_i$ and an isomorphism $\Phi:Q_1\to Q_2$, and then glue along $\Phi$ to produce $\hsm{C}=\hsm{C}_1\sqcup_\Phi\hsm{C}_2$. Let $\Pi:\hsm{C}_1\sqcup\hsm{C}_2\to\hsm{C}$ be the gluing map identifying every $x\in Q_1$ with $\Phi(x)\in Q_2$.

When can we conclude that $\hsm{C}$ is non-positively curved? Look at the link of the only vertex $v=\pi(v_1)=\pi(v_2)$: the cellular map $\Pi:\hsm{C}_1\sqcup\hsm{C}_2\to\hsm{C}$ induces a simplicial map $\pi:\link{v_1}{\hsm{C}_1}\sqcup\link{v_2}{\hsm{C}_2}\to\link{v}{\hsm{C}}$. Each of the $\pi\big|_{\link{v_i}{\hsm{C}_i}}$ is an isomorphism onto its image, and the only identifications are of the form $\pi(u_1)=\pi(u_2)$ for $u_1\in\link{v_1}{Q_1}$ and $u_2\in\link{v_2}{Q_2}$ satisfying $\link{v_1}{\Phi}(u_1)=u_2$. Thus, by the following lemma, in order for $\hsm{C}$ to be NPC, it suffices to show that each of the links $\link{v_i}{Q_i}$ is full in $\link{v_i}{\hsm{C}_i}$.
\begin{lemma}\label{gluing lemma}
Let $K_1$ and $K_2$ be flag complexes. Suppose that:
\begin{enumerate}
	\item $L_1 \subset K_1$ is a non-empty full sub-complex of $K_1$,
	\item $\varphi: L_1 \to K_2$ is a simplicial embedding, and --
	\item $\varphi(L_1)^{(1)}$ is full in $K_2^{(1)}$.
\end{enumerate}
Then $K=K_1 \sqcup_{\varphi} K_2$  is a flag complex. In particular, $K$ is a flag complex if both $L_1$ and $\varphi(L_1)$ are full sub-complexes of $K_1$ and $K_2$, respectively.
\end{lemma}
\begin{proof} Let $\Delta$ be a sub-complex of $K$ isomorphic to the $1$-skeleton of an $n$-dimensional simplex for some $n\geq 2$.

Let $L$ denote the image of $L_1$ in $K$ under the natural projection $\pi:K_1\sqcup K_2\to K$ and let $V=\Delta^{(0)}\minus L^{(0)}$. %Observe that $L$ is flag: $L$ is isomorphic to $L_1$, which is a full sub-complex in a flag complex -- hence is itself flag.

Suppose $V$ is empty, so that $\Delta^{(0)}\subseteq L$. Since $L_1$ is full in $K_1$, $L$ is full in $\pi(K_1)$. Thus every edge of $\pi(K_1)$ joining two vertices of $\Delta$ is and edge belonging to $L=\pi(K_1)\cap\pi(K_2)$. Hence all edges of $K$ joining vertices of $\Delta$ are edges of $\pi(K_2)$. Since $\pi(K_2)$ is flag, this means that $\Delta$ bounds an $n$-simplex in $\pi(K_2)$ and we are done.

Thus, we may assume $V$ is non-empty.

Suppose $\Delta\not\subset\pi(K_i)$ for $i=1,2$ and let $V_i=V\cap\pi(K_i)$. Since $\Delta\not\subset\pi(K_i)$ for both $i$, none of the $V_i$ is empty. But then the open simplex spanned by $V$ (it exists, since $V$ is a proper subset of $\Delta^{(0)}$) does not belong to either $K_i$ -- contradiction.

Thus, $\Delta\subset\pi(K_i)$ for at least one of $i=1,2$. Observe that $\pi\big|_{K_i}$ is an isomorphism onto its image (for both $i=1,2$). Since $K_i$ is full, so is $\pi(K_i)$, and $\Delta$ bounds a simplex in $\pi(K_i)\subset K$ -- we are done.
\end{proof}
\begin{corollary} Let $\Gamma$ be an oriented simple graph and let $d:F\to\NNhat$ be an admissible marking with support $F$. Then the cubical complex $X_\Gamma(F,d)$ is non-positively curved.
\end{corollary}
\begin{proof} Write $F=\{f_1,\ldots,f_r\}$ and $d(f_i)=d_i$ ($i=1,\ldots,r$). Start with $X_0=X_{\Gamma-F}$ and define inductively $X_{i+1}$ to be the result of attaching the space $Y_{f_i}$ (as described in the preceding paragraph) to $X_i$ using the attaching map $\Phi_{f_i}$.

Forming the space $X_1$ is a direct application of the last lemma, and we have that the complex $\link{\ast}{X_{\Gamma-F}}$ is a full sub-complex of $\link{\ast}{X_1}$.

To to see that it is possible to apply the lemma in the subsequent stages we use induction on $i$. Suppose that for some $j\geq 1$ the cubical complex $X_j$ is non-positively curved and that $\link{\ast}{X_{\Gamma-F}}$ is a full sub-complex in $\link{\ast}{X_j}$. We will prove that the same holds for $X_{j+1}$.

Let $f=f_{j+1}$. Since $\Sigma_f$ is full in $K_{\Gamma-F}$, we have that $S(\Sigma_f)$ is full in $S(K_{\Gamma-F})=\link{\ast}{X_{\Gamma-F}}$. Therefore, $S(\Sigma_f)$ is a full sub-complex of $X_j$, and we may apply the gluing lemma to the map $\varphi_f:S(\Delta_f)\to S(\link{f}{K_\Gamma})\join\link{\ast}{\PP{d(f)}}$ to conclude that $X_{j+1}$ is non-positively curved.

It remains to verify that $\link{\ast}{X_{\Gamma-F}}$ is full in $\link{\ast}{X_{j+1}}$. Since both complexes are simplicial flag complexes, it suffices to show that if $\sigma$ is an edge in $\link{\ast}{X_{j+1}}$ joining two vertices of $\link{\ast}{X_{\Gamma-F}}$, then $\sigma$ is an edge of $\link{\ast}{X_{\Gamma-F}}$.

If $\sigma\in\link{\ast}{X_j}$ then $\sigma\in S(K_{\Gamma-F})$, by the inductive hypothesis. However, $\sigma\in\link{\ast}{X_{j+1}}\minus\link{\ast}{X_j}$ connects two vertices of $S(K_{\Gamma-F})$ by construction, and we are done.
\end{proof}

\subsection{Finiteness properties} Once again, $\Gamma$ is an oriented simple graph and $d:E\Gamma\to\NNhat$ is an admissible marking.

Now we turn to the topological properties of the spaces $X_\Gamma(F,d)$. Keeping the same notation as before, we observe an important property of the gluing maps $\Phi_f$ ($f\in F$): positive vertices of $\link{\ast}{X_{\Gamma-F}}=S(K_{\Gamma-F})$ are identified with positive vertices of $\link{\ast}{Y_f}$, and the same is true for negative vertices. This means that the Morse functions we had defined on the complexes $\widetilde{X}_{\Gamma-F}$, $\widetilde{X}_\Gamma$ and $Y_f$ ($f\in F$) are all compatible, resulting in a Morse function $\widetilde{h}$ on $\widetilde{X}_\Gamma(F,d)$. Now the fact that the ascending and descending links in the spaces $\PP{d(f)}$ are trees comes into play one more time:

\begin{lemma}\label{same topology of links} Let $h:\widetilde{X}_\Gamma\to\RR$ be the standard Morse function of $A_\Gamma$ and let $\widetilde{h}:\widetilde{X}_\Gamma(F,d)\to\RR$ denote the Morse function corresponding to the homomorphism (also denoted by $\widetilde{h}:A_\Gamma(F,d)\to\ZZ$) sending all the generators $v\in V\Gamma$ and all $a_{f,i}$ and $s_f$ to $1\in\ZZ$. Then both the ascending and descending links of $\widetilde{h}$ are homotopy-equivalent to $K_\Gamma$.
\end{lemma}
\begin{proof} For the purpose of the proof we keep the notation from the previous paragraph, and recall that the ascending (and descending) link of $h$ is isomorphic to the simplicial complex $K_\Gamma$ -- the flag complex generated by the graph $\Gamma$.

To obtain the ascending link of $X_\Gamma(F,d)$ from that of $X_\Gamma$, for each $f\in F$ the sub-complex $\st{f}{K_\Gamma}=\link{f}{K_\Gamma}\join[0,1]$ of $K_\Gamma$ is replaced by the complex $\link{f}{K_\Gamma}\join\dnlink{\ast}{\PP{d(f)}}$, with all vertices of $\link{f}{K_\Gamma}$ staying put, while two of the vertices of the tree $T^+=\dnlink{\ast}{\PP{d(f)}}$ -- call them $\alpha$ and $\beta$ -- are glued to the vertices $\init f^+$ and $\term f^+$. The same (though with a different tree $T^-$) happens with the ascending links. If now $g$ is a strong deformation retraction of $T^+$ onto the geodesic path in $T^+$ joining $\alpha$ to $\beta$, then $g$ extends to a strong deformation retraction of $\link{f}{K_\Gamma}\join T^+$ onto $\link{f}{K_\Gamma}\join[0,1]$, as desired.
\end{proof}
Applying the tools of \cite{[Bestvina-Brady]}, we obtain the following result:
\begin{corollary}\label{kernels finitely presented} Let $\widetilde{h}$ be as defined in the preceding lemma. If the complex $K_\Gamma$ is simply connected, then $\ker\widetilde{h}$ is finitely presented.
\end{corollary}

\section{Upper Bound on the Dehn Function}\label{section:upper bounds} We first recall some facts regarding the proper tools for the computation of filling invariants.
\subsection{Admissible maps} We recall the definition of admissible maps and Dehn functions of groups from \cite{[SnowFlake]}. If $W$ is a compact $k-$dimensional manifold and $X$ a CW complex, an \textit{admissible map} is a continuous map $f:W\to X^{(k)}$ such that $f\inv (X^{(k)} - X^{(k-1)})$ is a disjoint union of open $k-$dimensional balls, each mapped by $f$ homeomorphically onto a $k-$cell of $X$.

If $\dim W=2$ and $f$ is admissible, we define the \textit{area} of $f$, denoted $\area{f}$, to be the number of disks in $W$ mapping to $2-$cells of $X$. This notion is useful due to the abundance of admissible maps:

\begin{lemma}[\cite{[SnowFlake]}, Lemma 2.3]\label{admissible}
Let $W$ be a compact manifold (smooth or piecewise-linear) of dimension $k$ and let $X$ be a CW complex. Then any continuous map $f:W\to X$ is homotopic to an admissible map. If $f(\partial W) \subset X^{(k-1)}$ then the homotopy may be taken rel $\partial W$.
\end{lemma}

Given a finitely presented group $G$, fix a $K(G,1)$ space $X$ with finite $2-$skeleton. Let $\widetilde{X}$ be the universal cover of $X$. If $f:\fat{S}^1\to\widetilde{X}$ is an admissible map, define the \textit{filling area} of $f$ to be the minimal area of an admissible extension of $f$ to $\fat{B}^2$:
\begin{displaymath}
	\text{FArea}(f) =\min\left\{\area{g}\,\bigg|\, g:\fat{B}^2\to\widetilde{X}\,,\; g\res{\partial \fat{B}^2} =f \right\}.
\end{displaymath}

Note that extensions exist since $\widetilde{X}$ is simply connected, and any extension can be made admissible by the lemma \ref{admissible}. The Dehn function of $X$ is defined to be
\begin{displaymath}
	\delta (n) = \sup\left\{
		\text{FArea}(f)\,\bigg|\, f:\fat{S}^1\to\widetilde{X}\,,\; \area{f}\leq n
	\right\}.
\end{displaymath}
Again, the maps $f$ are assumed to be admissible.

\subsection{Dimpling and Pushing} From now on, assume $\Gamma$ is an oriented graph and $d$ is an admissible marking. We denote $\displaystyle d_{max}=\max_{f\in F}d(f)$. To simplify notation, set $\widetilde{X}=\widetilde{X}_\Gamma(F,d)$.

Begin by refining the cell structure on $\widetilde{X}$ using $\zl$ to subdivide all cells intersecting this level set. Set this to be the cell structure on $\tilde X$ from now on. %Thus $\zl$ becomes triangulated by horizontal equilateral triangles labeled by the generators of $\ker(h)$ corresponding to horizontal diagonals of squares in $X_\Gamma(F,d)^{(2)}$.

As a result, a vertical square $\sigma^{(2)}\in\widetilde{X}^{(2)}$ with $\widetilde{h} (\sigma^{(2)}) =[-1,1]$ is expressed as a union of two isosceles right-angled triangles in $\widetilde{X}^{(2)}$, with common diagonal in $\zl$. Thus, $\widetilde{X}^{(2)} \setminus \zl$ consists of two types of vertical $2-$cells: squares and triangles.

Let $N(v)$ denote an open ball of radius $\frac{1}{4}$ about $v \in\widetilde{X}^{(0)}$. Consider the closed subspace $\widetilde{Y}$ obtained from $\widetilde{X}$ by removing all $N(v)$ for $v\notin\zl$.

We endow $\widetilde{Y}$ with the cellular structure inherited from $\widetilde{X}$: every cube in $\widetilde{X}$ undergoes truncation of all the relevant vertices, with a corner of an $i$-dimensional cube replaced by an $i-1$ spherical simplex. As a consequence, $\widetilde{Y}$ has two types of cells: {\it link-cells} coming from links of vertices and \textit{truncated cells} coming from the cells of $\widetilde{X}$. Since every link-cell $\sigma$ is associated with a unique vertex $v$ of $\widetilde{X}$, it will be convenient to abuse notation and say that $\sigma$ has height $\widetilde{h}(\sigma)=\widetilde{h}(v)$. Finally, let $\widetilde{Z}=\widetilde{Y}^{(2)}$.

Since $\widetilde{X}$ is CAT(0) there exists a constant $A>0$ such that for all $n\geq 0$ and every edge-loop $\ell$ of length at most $n$ in $\zl$, there exists a Van-Kampen diagram $V_\ell$ of area at most $An^2$ and of diameter at most $An$ for the corresponding trivial word in $\ker(\widetilde{h})$. The diagram $V_\ell$ defines a cellular filling $F_\ell:V_\ell\to\widetilde{X}^{(2)}$ of $\ell$.

Our goal is to produce an admissible filling of $\ell$ in $\zl$ of controlled area. We require the following ingredients:
\begin{enumerate}
	\item {\bf Dimpling Procedure: } Let $V_0$ be the set of all $v\in V_\ell^{(0)}$ with $h(F_\ell(v))\neq 0$. Consider $v\in V_0$. Let $n_v$ be the valence of $v$ in $V_\ell^{(1)}$. Let $N(v)$ be a disk neighbourhood of $v$ satisfying $F_\ell\left(N(v)\right)\subset N\left(F_\ell(v)\right)$ and $F_\ell\left(\bd N(v)\right)\subset\bd N(F_\ell(v))$. Thus, the choice of $N(v)$ produces an induced map of the cycle $\bd N(v)$ into $\link{F_\ell(v)}{\widetilde{X}}$.
	
	Let $V_\ell^\ast$ be the result of removing all the $N(v)$, $v\in V_0$, from $V_\ell$. Let $F_\ell^\ast$ be the restriction of $F_\ell$ to $V_\ell^\ast$.
	
	The plan is to construct an admissible extension $G_\ell$ of $F_\ell^\ast$ to $V_\ell$ -- possibly subdivided -- such that for every $v\in V_0$, $G_\ell(N(v))\subset\widetilde{Z}\cap N\left(F_\ell(v)\right)$, and such that the $\area{G_\ell\res{N(v)}}\leq B\cdot n_v$, where $B$ is a constant depending only on $\widetilde{X}$.
	\item {\bf Pushing Map: } We construct a continuous map $\hsm{P}:\widetilde{Z}\to\zl$ satisfying the following properties:
	\begin{itemize}
		\item[(P0)] $\hsm{P}$ fixes $\zl$ pointwise,
		\item[(P1)] $\hsm{P}$ maps truncated cells to vertices and/or edge-paths.
		\item[(P2)] There is a constant $C>0$ such that for every $H\in\ZZ$, $\hsm{P}$ maps link $2$-cells at height $H$ to unions of at most $C|H|^{d_{max}+1}$ triangles of $\zl$ whenever $d_{max}<\infty$, and at most $C^{|H|+1}$ if $d=\infty$.
	\end{itemize}
\end{enumerate}

\begin{figure}[ht]
    \includegraphics[width=1.0\textwidth]{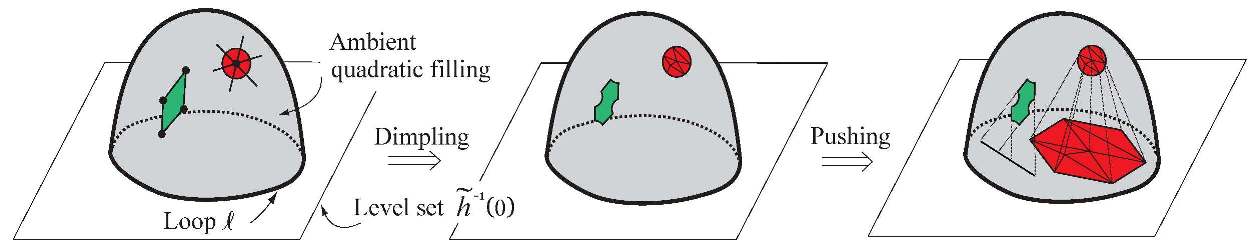}
    \caption{\tiny Trucated 2-cell (green, left) collapse under pushing, while link 2-cells (red, right) blow up.\normalsize\label{fig:dimpling and pushing}}
\end{figure}

Assume we have both of the above and consider $\hsm{P}\circ G_\ell$. This is an admissible filling map for the edge-loop $\ell$, and we may compute the area of this filling:
\begin{itemize}
	\item[-] Looking at $V_\ell$ we have the inequality:
	\begin{displaymath}
		\sum_{v\in V_\ell^{(0)}}n_v=2\left|V_\ell^{(1)}\right|=\sum_{f\in V_\ell^{(2)}}n_f+n\leq 4\left|V_\ell^{(2)}\right|+n\,,
	\end{displaymath}
	where $n_f$ is the number of edges bounding a face. Thus, removing the boundary vertices from the count results in:
	\begin{displaymath}
		\sum_{v\in V_0}n_v\leq 4\area{F_\ell}\leq 4An^2\,.
	\end{displaymath}
	\item[-] Hence the number of edges of $V_\ell^\ast$ is bounded by the number of $2-$cells of $V(\ell)$ up to multiplication by a constant.
	\item[-] Then the area of $\hsm{P}\circ G_\ell$ is estimated as follows:
	\begin{eqnarray*}
		\displaystyle
		\area{\hsm{P}\circ G_\ell}
			&	=	& \area{\hsm{P}\circ G_\ell\res{V_\ell^\ast}}+\sum_{v\in V_0}\area{\hsm{P}\circ G_\ell\res{N_v}} \\
			&	\leq	&	0+\sum_{v\in V_0}\area{G_\ell\res{N(v)}}C\left|\widetilde{h}(F_\ell(v))\right|^{d+1} \\
			&	\leq	&	\sum_{v\in V_0} Bn_v\cdot Cn^{d+1}\\
			&	\leq	&	BC n^{d+1}\sum_{v\in V_0} n_v\\
			&	\leq	& 4ABC n^{d+3}\,.
	\end{eqnarray*}
	for $d_{max}<\infty$ and similarly for $d=\infty$, to produce an exponential bound.
\end{itemize}
Thus, up to constructing dimpling and pushing for our particular space $\widetilde{X}$, we have proved:
\begin{theorem}[Main theorem, upper bound]\label{thm:upper bound} Let $\Gamma$ be the $1$-skeleton of a simply-connected simplicial flag complex $K$, and let $d:E\Gamma\to\NNhat$ be an admissible marking. Then the kernel $\ker_\Gamma(F,d)$ of the height function $\widetilde{h}:A_\Gamma(F,d)\to\ZZ$ satisfies a polynomial isoperimetric inequality of degree $(d_{max}+3)$ when $d_{max}<\infty$, and an exponential one when $d_{max}=\infty$.\ep
\end{theorem}
An exponential upper bound for the general case is easily achieved through a result of Gersten and Short (\cite{[Gersten-Short]}, theorem B), using the fact that $A_\Gamma(F,d)$ is biautomatic (by the result of Niblo and Reeves \cite{[Niblo-Reeves-biautomatic]}). Our approach, however, has the advantage of producing a uniform argument for both the polynomial and the exponential cases.
\begin{prop}[Dimpling]
There exist a constant $B>0$ depending only on the graph $\Gamma$ and the marking $d$, and a subdivision $V'_\ell$ of $V_\ell$ and an extension $G_\ell:V'_\ell \to \widetilde{Z}$ of $F_\ell\res{V_\ell^\ast}$ such that $G_\ell(N(v))\subset\widetilde{Z}\cap N(F_\ell(v))$ and $\area{G_\ell\res{N(v)}}\leq B n_v$.
\end{prop}
\begin{proof} Let $B=\delta(2m+1)$ where $m=\mathrm{diam}\,\link{\ast}{X}$ and $\delta$ is the Dehn function of $L=\link{\ast}{X}$. Note that each link of the form $\link{F_\ell(v)}{\widetilde{X}}$ is naturally identified with $L$ through the covering map $\widetilde{X}\to X$.

The proof follows one of the standard arguments showing that a finite simply connected complex has a linear Dehn function.

By a quarter-disk we shall mean a right-angled sector of the unit disk, triangulated so that the cone point is a vertex and the circular boundary arc is an edge of the triangulation.

Fix a base point $v_0\in L$. For every $w\in L^{(0)}$ let $p_w$ be an edge-path from $v_0$ to $w$ and let $\bar p_w$ denote the reverse path. Since $L$ is simply connected, for every directed edge $e$ in $\link{\ast}{X}$, there is a combinatorial map $G_e$ from a quarter-disk $\Delta$ to $L^{(2)}$ such that $G_e$ maps $\bd\Delta$ to the path $p_{\init e}e\bar p_{\term e}$, with the circular arc (positively oriented) mapping onto $e$. We can choose $G_e$ to have area at most $B$.

Now, for any $v\in V_0$, consider the edge-loop $\ell_v=(e_0,\ldots,e_{n_v-1})$ in $L$ defined by $F_\ell\res{\bd N(v)}$. Consider $N(v)$ as the result of gluing quarter-disks $\Delta_i$, $i=0,\ldots,n_v-1$ in cyclical order modulo $n_v$. Subdivide the $\Delta_i$ appropriately for each $\Delta_i$ to support the map $G_{e_i}$. The resulting glued map $G'_v:N(v)\to L^{(2)}$ lifts to a map $G_v:N(v)\to N(F_\ell(v))\cap\widetilde{Z}$. By construction, the area of $G_v$ is at most $Bn_v$.

$G_\ell$ is the result of gluing $F_\ell^\ast$ with the maps $G_v$, $v\in V_0$. This completes the proof.
\end{proof}
Now we construct the pushing map.
\begin{prop}[Pushing] Let $\Gamma$ be the $1$-skeleton of a finite simply-connected simplicial flag $2$-complex $K$. Let $d:E\Gamma\to\NNhat$ be an admissible marking. Then there is a map $\hsm{P}:\widetilde{Z}\to\zl$ satisfying conditions (P0)-(P2).
\end{prop}
\begin{proof} Observe that if $\sigma$ is a vertex or truncated $1$-cell in $\widetilde{Z}$ then there is a unique $1$-cell of $\widetilde{X}$ containing $\sigma$, labeled by a generator of $A_\Gamma(F,d)$; denote this generator by $\lambda(\sigma)$.

We construct $\hsm{P}$ inductively on the skeleta of $\widetilde{Z}$.

{\bf Defining $\hsm{P}$ on $\widetilde{Z}^{(0)}$. } Given $w\in\widetilde{Z}^{(0)}\setminus\zl$, set $\hsm{P}w=\ell_w\cap\zl$, where $\ell_w$ is the unique vertical geodesic containing $w$ and all of whose edges are labeled by the same generator as the edge of $\widetilde{X}$. Observe that $\ell_w$ is the unique vertical periodic geodesic passing through $w$.

{\bf Extending $\hsm{P}$ to $\widetilde{Z}^{(1)}$. } Let $P_f$ denote the set of generators we have used for constructing the perturbing groups $S_{d(f)}$.

Note that $\hsm{P}$ maps both endpoints of a truncated $1$-cell to the same destination. Therefore it is possible to extend $\hsm{P}$ so that it is constant along such cells.

The situation with link $1$-cells is different. Given a link $1$-cell $\sigma\in\widetilde{Z}^{(1)}$ at height $n\neq 0$ with endpoints $u,v$ labeled $a$ and $b$ respectively, there are several possibilities:
\begin{enumerate}
	\item \underline{$a$ and $b$ commute.}\; In this case, the lines $\ell_u$ and $\ell_v$ span a periodic vertical $2$-flat $F(\sigma)$, so that the geodesic joining $\hsm{P}(u)$ with $\hsm{P}(v)$ is contained in $\zl\cap F(\sigma)$. We let $\hsm{P}$ map $\sigma$ onto $[\hsm{P}u,\hsm{P}v]$ by radial projection.
	\item \underline{$a,b\in P_f$ for some $f\in F$.}\; Identify the vertices of $\widetilde{X}$ with elements of $A_\Gamma(F,d)$. Thus, $\hsm{P}u\in \ker_\Gamma(F,d)$. By proposition \ref{fans are embedded}, the fan $\mathrm{Fan}\left(a^n,b^n\right)$ embeds in $\tilPP{d(f)}$ and hence also in $\widetilde{X}$. We identify the fan with its image in $\widetilde{X}$. Denote the translate $\hsm{P}u\cdot\mathrm{Fan}\left(a^n,b^n\right)$ of this fan by $\mathrm{Fan}(\sigma)$. Observe that the sides of $\mathrm{Fan}(\sigma)$ are contained in $\ell_u$ and $\ell_v$, while the edge-rim is a reduced edge-path joining $\hsm{P}u$ with $\hsm{P}v$. We let $\hsm{P}$ map the cell $\sigma$ onto this path. Note that when $a$ and $b$ commute this coincides with the construction in the first case.
\end{enumerate}

{\bf Extension to truncated $2$-cells in $\widetilde{Z}$. }
Let $\sigma$ be a truncated $2$-cell. We claim that $\hsm{P}\bd\sigma$ is a tree (contained in the $1$-skeleton of $\zl$. This will enable extending $\hsm{P}$ continuously over such cells so that $\hsm{P}\sigma=\hsm{P}\bd\sigma$.

There are two cases to consider:
\begin{enumerate}
	\item If $\bd\sigma$ is labeled by a pair of commuting generators, then $\hsm{P}\bd\sigma$ is contained in the image under $\hsm{P}$ of the top (or bottom) link $1$-cell of $\bd\sigma$, which is a geodesic.
	\item Otherwise, all labels on $\bd\sigma$ belong to $P_f$ for some $f\in F$. In this case, the fans corresponding to the link $1$-cells of $\sigma$ are all contained in the same $\ker_\Gamma(F,d)$-translate of an isometrically embedded copy of $\tilPP{d(f)}$, whose intersection with $\zl$ is a tree. Therefore, $\hsm{P}\bd\sigma$ is contained in this tree, as claimed.
\end{enumerate}
Note that this also covers the case of truncated isosceles $2$-cells: $\hsm{P}$ maps such cells onto their intersection with $\zl$. Clearly, it can be arranged for $\hsm{P}$ to fix its image pointwise.

{\bf Extension to link $2$-cells in $\widetilde{Z}$. }  Observe that $\link{\ast}{X}$ is the union of $S(K_{\Gamma-F})$ with all the joins $S(\link{f}{K_\Gamma})\join\link{\ast}{\PP{d(f)}}$, $f\in F$. Since the $\link{\ast}{\PP{d(f)}}$ are all $1$-dimensional, every $2$-cell in $\link{\ast}{X}$ has at least one vertex corresponding to a generator from $V\Gamma$ and necessarily commuting with the generators corresponding to the other vertices.

\begin{figure}[ht]
    \includegraphics[width=\textwidth]{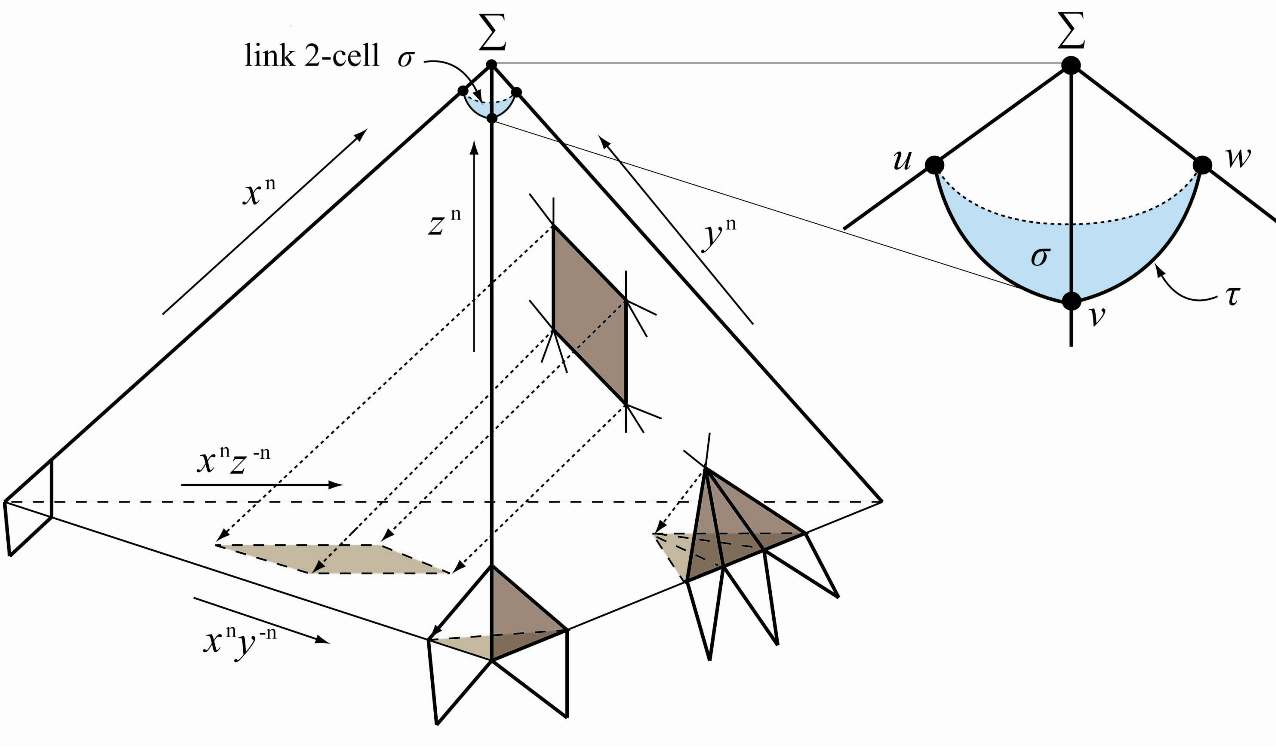}
    \caption{\tiny Link 2-cells are mapped to horizontal filling. \normalsize\label{fig:projection}}
\end{figure}

Finally, let $\sigma$ be a link $2-$cell at height $n$ having vertices $u,v,w$ with labels $x, y$ and $z$. Let $\Sigma$ be the vertex of $\widetilde{Z}$ whose link contains $\sigma$. Due to the structure of the link described above, there are two cases to consider:
\begin{enumerate}
	\item \underline{$x,y,z$ commute.}\; The lines $\ell_u$, $\ell_v$ and $\ell_w$ span an isometrically embedded $3$-flat $F$. This defines a flat tetrahedron with base $F\cap\zl$, vertex $\Sigma$, and vertical $2$-faces. The $1$-cells of $\sigma$ are mapped to the boundary of the base of the tetrahedron, so we may extend $\hsm{P}$ to $\sigma$ by projecting $\sigma$ onto the base of the tetrahedron in the obvious way, and we see that $\hsm{P}\sigma$ will have an area that is quadratic in the height $\widetilde{h}(\sigma)$ -- which is less than any of the upper bounds claimed in the statement of our proposition.
	\item \underline{$y,z\in P_f$ for some $f\in F$ and $x$ commutes with $P_f$.}\; Let $n = \widetilde{h}(\sigma)$ and let $\tau$ be the link $1$-cell of $\sigma$ joining $v$ and $w$. The vertical sub-complex $F=\mathrm{Fan}(\tau)\cap \widetilde{h}\inv([0,n])$ of $\widetilde{X}$ consists of vertical $2$-cells. Define a map $\pi:F^{(0)}\to\zl$ sending every vertex $p$ to $p\cdot x^{-\widetilde{h}(p)}$. Since $x$ commutes with every generator in $P_f$, $\pi$ extends to a piecewise-linear injective map -- denote it by $\pi$ too -- of $F$ to $\zl$. Observe that $\pi$ maps $\bd F$ homeomorphically onto $\hsm{P}\bd\sigma$ fixing $\hsm{P}\tau$ pointwise (see figure \ref{fig:projection}). Thus $\pi(F)$ is contractible and $\hsm{P}$ can be extended to an admissible map over the whole of $\sigma$. The area of $\pi(F)$ equals the area of $F$, which is polynomial in $n$ of order at most $d(f)+1\leq d_{max}+1$ when $d(f)<\infty$ (propositions \ref{cor:polynomial growth of rims} and \ref{lemma:ascending fans are polynomial}), and in general at most exponential in $n$, by remark \ref{remark:exp upper bound}.
\end{enumerate}
\end{proof}
\begin{remark} In fact, the above argument can be extended to demonstrate that $\hsm{P}$ can be extended to the whole of $\widetilde{Y}$. This could be useful in consideration of higher-dimensional Dehn functions of PRAAG kernels.
\end{remark}

\section{The Perturbed $1$-orthoplex}\label{section:orthoplex}
Here we study the perturbed $1$-orthoplex (compare with \cite{[ABDDY]}, definition 5.2) groups $A_\Delta(d)$ and $\ker_\Delta(d)$ for $1<d=d(e)\in\NNhat$ where $\Delta$ is the graph in figure \ref{fig:orthoplex} and $d\in\NNhat$ marks the edge $e$, all the other edges marked with zeros. We suppress all mention of $\Delta$ and $d$ in the notation wherever possible unless there is a risk of confusion.

\begin{figure}[t]
    \includegraphics[width=.8\textwidth]{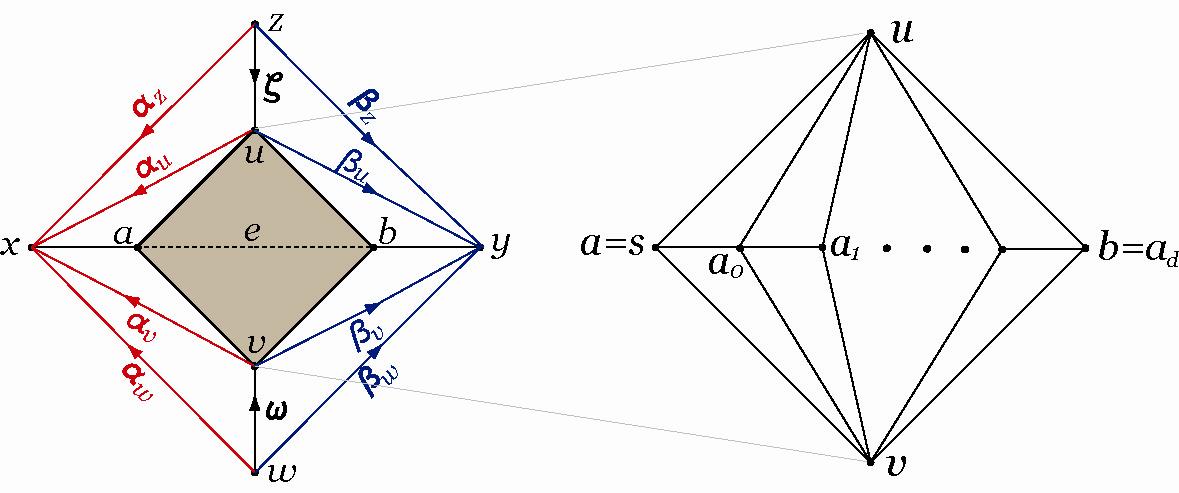}
    \caption{\tiny The perturbed $1$-orthoplex $\Delta$: all edges except $e$ in this diagram (left) represent standard generators of $\ker_\Delta(d)$ with $1<d=d(e)\leq\infty$. On the right hand side we see the finer structure of the descending link of the vertex in the perturbed complex $X_\Delta(d)$.\normalsize\label{fig:orthoplex}}
\end{figure}

\subsection{Fans revisited}
Given $n\in\NN$ consider the fan:
\begin{equation}
	F_n=\left\{\begin{array}{rl}
		\mathrm{Fan}(s^n,a_d^n)\,,& d<\infty\\
		\mathrm{Fan}(t^n,a_5^n)\,,& d=\infty
	\end{array}\right.
\end{equation}
In both cases, $F_n$ is the concatenation of `elementary' fans -- fans whose top vertex is adjacent to exactly one $2$-cell (in that fan). $F_n$ in hand, we chop off the vertices of $F_n$ at height $-1$ and the adjoining cells to produce a diagram bounded by the two ascending words and the edge-rim of $F_n$. Denote this diagram by $F'_n$.

\begin{figure}[ht]
    \includegraphics[width=.3\textwidth]{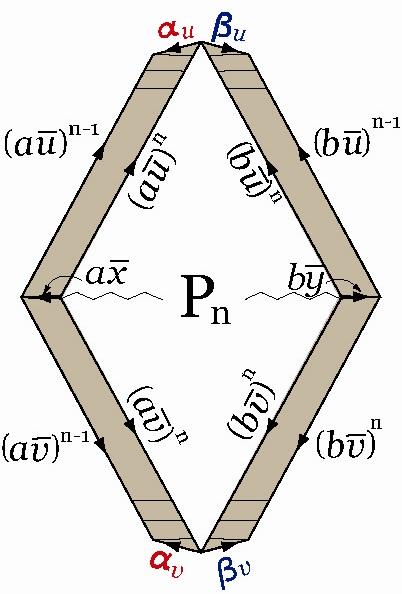}
    \caption{\tiny The diagram $P_n$ and some neighboring corridors in $\zl$.\normalsize\label{fig:pn}}
\end{figure}

Since both generators $u,v\in V\Delta$ commute with every element of $S_d$, replacing every label $\lambda$ on the edges of $F'_n$ by $\lambda\bar u$ produces an embedded diagram in $\zl$. The same holds if $\lambda$ is replaced by $\lambda \bar v$. The resulting two diagrams in $\zl$ intersect along the arc defined by the edge-rim of $F_n$, so their union is an embedded disk diagram $P_n$ in $\zl$ -- see figure \ref{fig:pn}.

An alternative way of obtaining $P_n$ uses the pushing map. For any vertex $O$ of $\widetilde{X}$, If $p,q,r$ are vertices in $L=\dnlink{O}{\widetilde{X}}$ bounding a $2$-simplex, denote this (closed) simplex by $\tri{O}{pqr}$ (we suppress the plus signs on all labels).

Let $\tri{O}{aub}$ denote the union of all triangles $\tri{O}{puq}$ where the edges $(p,q)$ run over the edges of $\dnlink{\ast}{\PP{d}}$ (see figure \ref{fig:orthoplex}, right-hand side). Define $\tri{O}{avb}$ analogously, and set $P(O)$ to be the image of $D(O)=\tri{O}{aub}\cup\tri{O}{avb}$ under pushing. Then $P(O)$ is an embedded copy of $P_n$ in $\zl$ for $n=\widetilde{h}(O)$ -- see figure \ref{fig:projection} and the discussion of the main case in the construction of the pushing map.

By pushing and by the results of section \ref{section:the groups S_d}, we now have
\begin{equation}
	\area{P_n}\asymp\left\{\begin{array}{rl}
		n^{d+3}\,,& d<\infty\\
		e^n\,,& d=\infty
	\end{array}\right.
\end{equation}
In the case $d=\infty$, note that the fan $F_n$ contains the fan $\mathrm{Fan}(a_1^{n-1},a_3^{n-1})$, and therefore has an exponential rim, as desired (see figure \ref{fig:exp_fans}).

\begin{figure}[ht]
    \includegraphics[width=.6\textwidth]{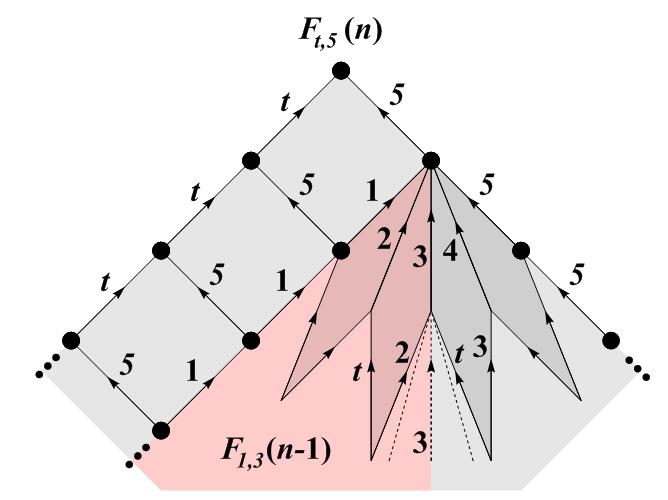}
    \caption{\tiny The fan $\mathrm{Fan}(a_1^{n-1},a_3^{n-1})$ embedded in $F_n=\mathrm{Fan}(t^n,a_5^n)$ (case $d=\infty$).\normalsize\label{fig:exp_fans}}
\end{figure}

\subsection{The exponential case}
This case is settled by looking at the diagrams $P_n$: for every $n$, the perimeter of $P_n$ equals $4n$, while the area grows exponentially.

\subsection{The polynomial case}
In this case, the diagrams $P_n$ have insufficient area for producing the desired bound. This is where the orthoplex structure becomes handy.

%DESCRIPTION OF LOOP AND TENT, CLAIM RESULT
It will be convenient to use the notation $\alt{k}{\alpha}{\beta}$ for the alternating word of length $k$ in the letters $\alpha$ and $\beta$, $\alpha$ coming first. Clearly,
\begin{displaymath}
	\alt{k+1}{\alpha}{\beta}=\alpha\alt{k}{\beta}{\alpha}\,,\quad
	\alt{2k}{\alpha}{\beta}=(\alpha\beta)^k\,.
\end{displaymath}
Consider the loops $\ell_n\subset\zl$ (based at the origin) defined for $n=2k$, $k\in\NN$ by the words:
\begin{equation}
	\ell_n=\alt{n}{\bar\beta_z}{\bar\alpha_w}\cdot\alt{n}{\beta_w}{\alpha_z}\cdot
		\alt{n}{\bar\alpha_w}{\bar\beta_z}\cdot\alt{n}{\alpha_z}{\beta_w}\,.
\end{equation}
Figure \ref{fig:tent} demonstrates a filling of $\ell_n$ in $\widetilde{X}$ composed of four flat vertical right-angled triangles bounded (in addition to the four segments defining $\ell_n$) by the ascending words
\begin{equation}
	 \alt{n}{y}{x}\,,\quad \alt{n}{z}{w}\,,\quad
	 	 \alt{n}{x}{y}\,,\quad \alt{n}{w}{z}\,.
\end{equation}
Denote the union of these triangles by $T_n$ (we fondly refer to it as the `$n$-th tent') and let $O_n$ be the top vertex of $T_n$.

\begin{figure}[hbt]
    \includegraphics[width=.8\textwidth]{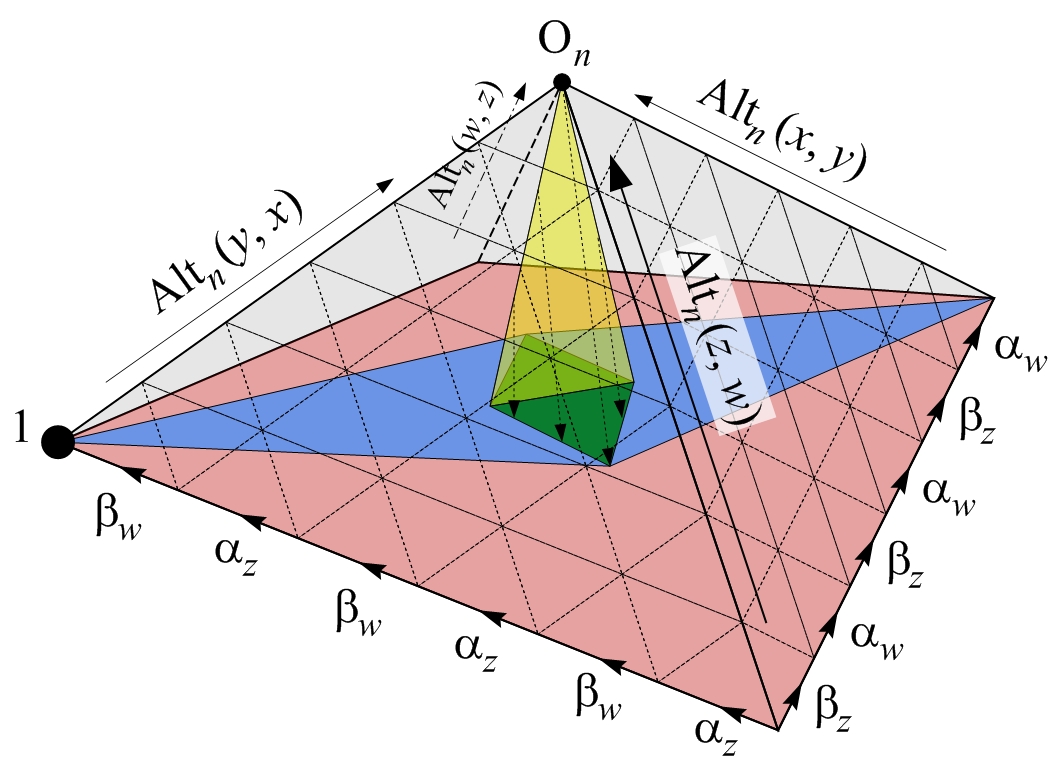}
    \caption{\tiny The tent $T_n$ with its faces subdivided (and the diagram $P(O_n)$ -- see below -- obtained by pushing (away from $O_n$) and situated in $\zl$ under the tent). Observe how the commutation relations among $x,z,y,w$ serve to form the faces of $T_n$.\normalsize\label{fig:tent}}
\end{figure}

%DESCRIPTION OF DIAGRAMS Pn,Qn,Rn
We remark that $T_n$ is a quadratic filling of the loop $\ell_n$ in $\widetilde{X}^{(2)}$, so it is natural to ask what the result of our dimpling-and-pushing procedure from the preceding section might look like.

We shall now construct a Van-Kampen diagram $R_n$ filling the loop $\ell_n$ in $\zl$, having the property that $\area{R_n}\asymp n^{d+3}$. The rest of the section will be devoted to a proof that $R_n$ is embedded, by studying the relationship between $T_n$ and $R_n$ using the pushing map $\hsm{P}$.

%CONSIDER THE DESCENDING LINK AND PUSHED DOUBLE FANS P_n
Figure \ref{fig:quen} details the labeling along the boundary of $P_n$ (in the case $d<\infty$), some adjacent corridors in $\zl$ and shows how to form a diagram $Q_n$ out of copies of $P_1,\ldots,P_n$ using those corridors. As a result we obtain $\area{Q_n}\asymp n^{d+2}$.

\begin{figure}[ht]
    \includegraphics[width=.8\textwidth]{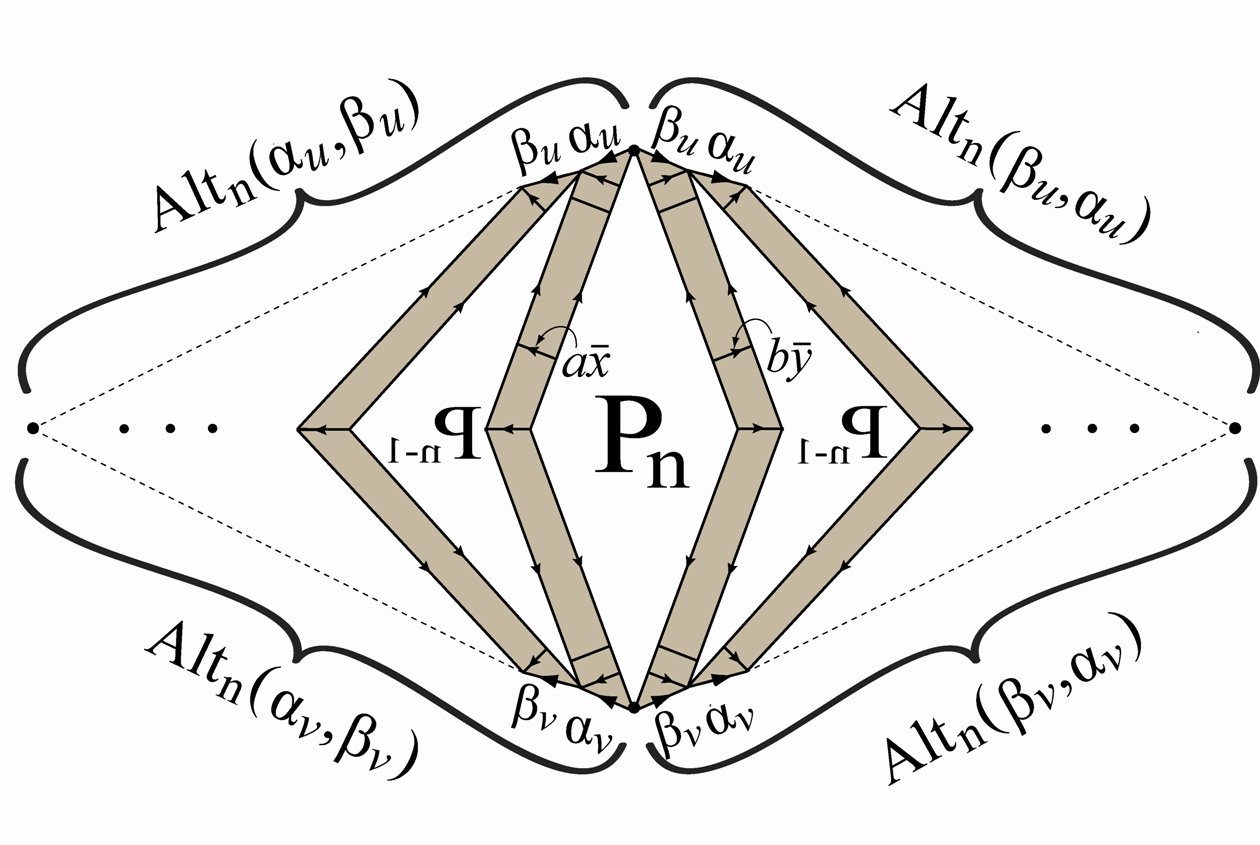}
    \caption{\tiny The diagrams $P_1,\ldots,P_n$ together with some surrounding corridors enable the construction of $Q_n$, with $\area{Q_n}\asymp n^{d+2}$.\normalsize\label{fig:quen}}
\end{figure}

The diagram $R_n$ is then constructed for {\it even} $n$ using $\zeta$- and $\omega$-corridors as shown on figure \ref{fig:rn}. As before, we have $\area{R_n}\asymp n^{d+3}$, and we see that $R_n$ defines a horizontal filling of the loop $\ell_n$ for all even $n$.

\begin{figure}[h]
    \includegraphics[width=.9\textwidth]{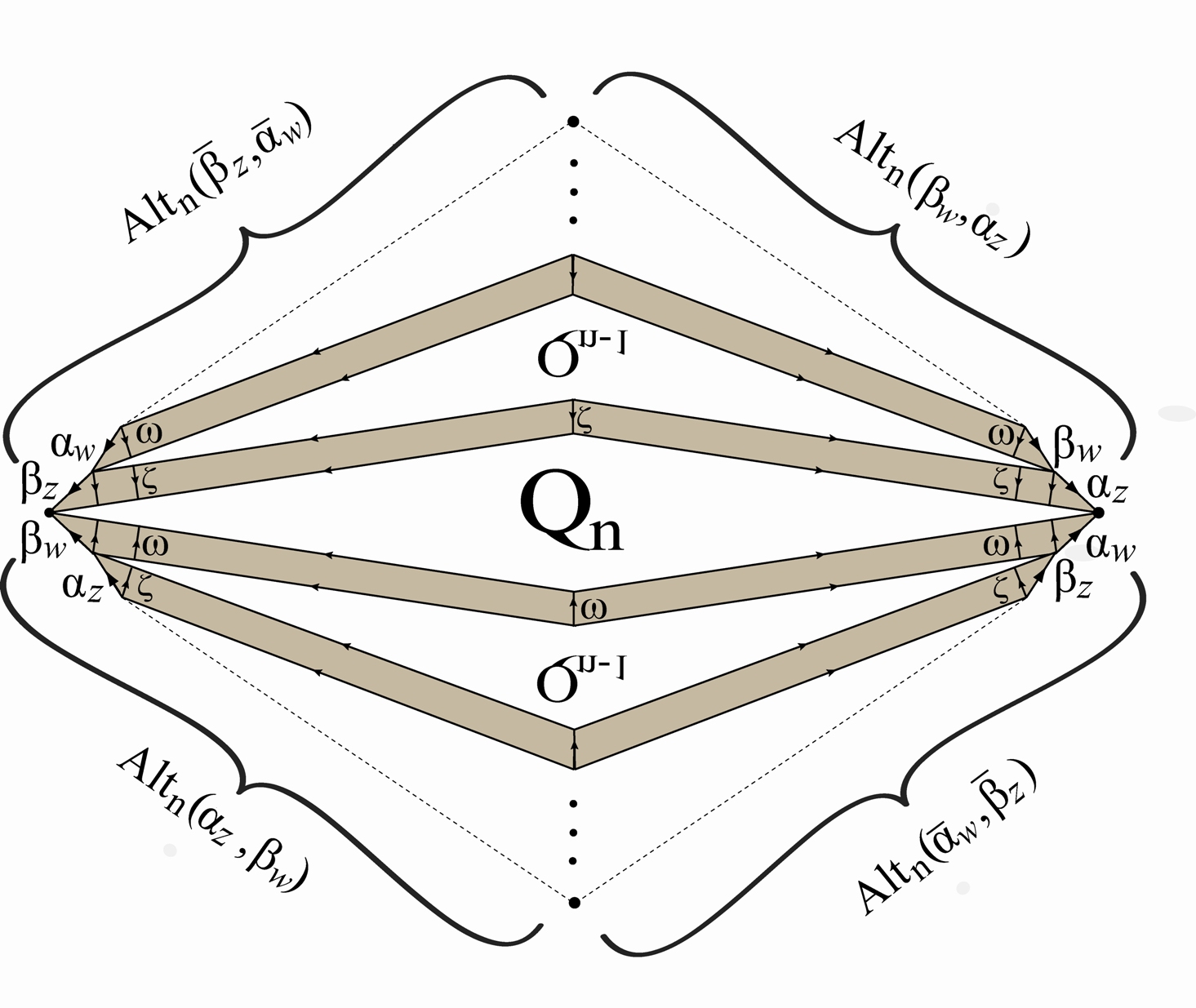}
    \caption{\tiny Gluing the $Q_n$ of various sizes into a diagram $R_n$ of perimeter $4n$ and area$\;\asymp n^{d+3}$. Note how the alternating patterns of reversing the orientation (from the construction of $Q_n$) repeats itself here, too.\normalsize\label{fig:rn}}
\end{figure}
%EXPLAIN WHY R_n IS EMBEDDED BY USING HYPERPLANES
Note that the construction of $Q_n$ and $R_n$ follows the guidelines of the construction in section 2.5.2 of \cite{[Brady-Blue-Book]}.
\begin{remark} Observe how the gluing procedure applied to $P_1,\ldots,P_n$ and then $Q_1,\ldots,Q_n$ {\it cannot} be repeated for $R_1,\ldots,R_n$ to produce a (useful) diagram of even higher area. Even forgetting about the upper bound on the Dehn function from the preceding section, observe that all four `boundary letters' are used on either side (left/right) of $R_n$, so repeating the procedure would necessitate the introduction of a new vertex into $\Delta$ corresponding to a generator which commutes with all four letters $x,y,z,w$. However, this alteration will result in forming an alternative filling of quadratic area for the loop $\ell_n$, defeating the purpose of the construction.
\end{remark}

Going back to the view of $P_n$ as the result of pushing $D(O_n)$, consider the following:
\begin{lemma} Let $O\in\widetilde{X}^{(0)}$ and $t\in\{x,y\}$. Denote $\tau=x\bar a$ if $t=x$ and $\tau=y\bar b$ if $t=y$. Let $H$ be the unique hyperplane of $\widetilde{X}$ separating $O$ from $Ot$. Then $H$ separates $P(O)$ from $P(Ot)$ and, moreover, $H$ intersects $\zl$ bisecting a $\tau$-corridor joining $P(O)$ to $P(Ot)$.
\end{lemma}
\begin{proof} Without loss of generality, $t=x$. By construction of $P(O)$, any vertex of $P(O)$ can be joined to $O$ by edges of $\widetilde{X}$ none of which is labeled $x$. The same holds for $P(Ox)$. Therefore, since $H$ separates $O$ from $Ox$, it also separates $P(O)$ from $P(Ox)$.

Consider the fan $Fan(b^{n}, a^{n})$ of $O$ and another fan $Fan(a^{n+1}, b^{n+1})$ of $Ox$. Since $x$ commutes with $a$ and $b$, these two fans can be joined by a vertical $x$-corridor. Moreover $u$ commute with $x$ and hence the definition of $P$ extends to this $x$-corridor to produce $(x \bar a)$-corridor in $\zl$. The same argument works for $v$. All these corridors are bisected by $H$.
\end{proof}

We proceed to enumerate the vertices of $T_n$ (recall $n$ is even!) as follows (consider the view of figure \ref{fig:tent} from above):
\begin{enumerate}
	\item Set $O_{0,0}=O_n$;
	\item For $i\geq 0$, let $O_{-i,0}=O_n\cdot\alt{i}{\bar x}{\bar y}$ and $O_{i,0}=O_n\cdot\alt{i}{\bar y}{\bar x}$;
	\item For $j\geq 0$, let $O_{i,j}=O_{i,0}\cdot\alt{j}{\bar z}{\bar w}$ and $O_{i,-j}=O_{i,0}\cdot\alt{j}{\bar w}{\bar z}$.	
\end{enumerate}
Thus the vertices of $T_n$ are enumerated as $\left\{O_{i,j}\right\}_{|i|+|j|\leq n}$. For a fixed $|j|\leq n$, let $S_j=\left\{O_{i,j}\right\}_{|i|\leq n-|j|}$.

We conclude from the last lemma that, given $1\leq m\leq n$, each of the two copies of $Q_m$ in $1\cdot R_n\subset\zl$ equals one of the unions $\displaystyle\bigcup_{O\in S_{\pm m}}P(O)$, with the addition of the appropriate corridors, with all the $P(O)$ separated from each other by hyperplanes of $\widetilde{X}$. We conclude that all the copies of $Q_1,\ldots,Q_n$ is $R_n$ are embedded. Denote the copy of $Q_m$ corresponding to $S_{\pm(n-m)}$ by $Q(O_{0,\pm(n-m)})\cong Q_m$.

The following lemma allows us to re-iterate the same reasoning and conclude that $R_n$ is embedded:
\begin{lemma} Let $O\in\widetilde{X}^{(0)}$ and $t\in\{z,w\}$. Denote $\tau=\zeta$ if $t=z$ and $\tau=\omega$ if $t=w$. Let $H$ be the unique hyperplane of $\widetilde{X}$ separating $O$ from $Ot$. Then $H$ separates $Q(O)$ from $Q(Ot)$ and, moreover, $H$ intersects $\zl$ bisecting a $\tau$-corridor joining $Q(O)$ to $Q(Ot)$.
\end{lemma}
\begin{proof} The proof is analogous to the proof of the last lemma.
\end{proof}
To summarize the results of this section, we have:
\begin{prop}\label{embedded Rn} Let $\Gamma$ be a directed simple graph with an admissible marking $d:E\Gamma\to\NNhat$. Suppose $\Gamma$ contains a $1$-orthoplex $\Delta$ as a full subgraph and $d$ vanishes on all edges of $\Delta$ except, perhaps its interior edge $e$. Then every level set of $\widetilde{X}_\Gamma(F,d)$ contains a family $(R_n)_{n=1}^\infty$ of embedded combinatorial disks with $\left|\bd R_n\right|\asymp n$ and $\area{R_n}\asymp n^{d(e)+3}$ when $1<d<\infty$, and $\area{R_n}\asymp e^n$ for $d=\infty$.
\end{prop}
Together with the upper bound from the previous section we now have:
\begin{corollary}\label{Dehn fncn of the orhtoplex} Let $K=\ker_\Delta(d)$ be the kernel of the height function on the perturbed $1$-orthoplex group $A_\Delta(d)$. Then $\delta_K(n)\asymp n^{d+3}$ for $1<d<\infty$, and $\delta_K(n)\asymp e^n$ for $d=\infty$.
\end{corollary}
\begin{proof} Since the level sets of $\widetilde{X}_\Delta(d)$ are contractible (descending and ascending links are contractible) and $2$-dimensional, proposition \ref{embedded Rn} implies a lower a bound of $n^{d+3}$ for the Dehn function of the kernel $\ker_\Delta(d)$ when $d$ is finite.
\end{proof}

\section{The Main Example}
Let $C$ be the boundary $4$-cycle of the $1$-orthoplex $\Delta$. Let $\Sigma$ be the double of the $1$-orthoplex $\Delta$ along $C$: we set $\Sigma=\Delta\cup_\lambda\Delta'$ where $\Delta$ is as before, $\Delta'$ is a copy of $\Delta$ with a vertex $v'$ and edge $f'$ for every $v\in V\Delta$ and $f\in E\Delta$, and $\lambda$ identifies $C$ with $C'$ according to the rule $\lambda:\sigma\mapsto\sigma'$ for every $\sigma\in C$.

Let $A_\Sigma(d)$ denote the PRAAG obtained by marking both the interior edge $e$ of $\Delta$ and its double $e'$ with $1<d\in\NNhat$ while leaving all other edges unmarked (marked by $0$).

Since $C$ is a full sub-complex of $\Delta$, we may apply the gluing lemma \ref{gluing lemma} to conclude that $A_\Sigma(d)=A_\Delta(d)\join_{A_C}A_\Delta(d)$ with $A_\Delta(d)$ embedding in $A_\Sigma(d)$ while retaining the same height function. In particular we have an inclusion of kernels $\ker_\Delta(d)<\ker_\Sigma(d)$.

\begin{theorem}\label{Dehn function of the double} The group $K=\ker_\Sigma(d)$ is of type $F_2$ but not $F_3$, and its Dehn function is given by: \[\delta_K(n)\asymp\left\{\begin{array}{rl} n^{d+3}, & d<\infty\\ e^n\,, & d=\infty\,.\end{array}\right.\]
\end{theorem}
\begin{proof} The homotopy type of the ascending/descending link is that of the $2$-sphere, by lemma \ref{same topology of links}. By theorem 2.4.2 of \cite{[Brady-Blue-Book]}, $\ker_\Sigma(d)$ has $F_2$ but not $F_3$.

By theorem \ref{thm:upper bound}, $\ker_\Sigma(d)$ has a polynomial isoperimetric inequality of degree $d+3$ when $d$ is finite, and exponential isoperimetric inequality when $d=\infty$.

To obtain the lower bound, we consider the retraction $\rho:A_\Sigma(d)\to A_\Delta(d)$ induced from retracting $\Sigma$ onto $\Delta$ by crushing the double in the obvious way.

Since $\rho$ commutes with the height functions, restricting $\rho$ to $\ker_\Sigma(d)$ produces a retraction $\bar\rho:\ker_\Sigma(d)\to \ker_\Delta(d)$. The existence of $\bar\rho$ implies $\delta_{\ker_\Delta(d)}\preceq\delta_{\ker_\Sigma(d)}$, which completes the proof after applying corollary \ref{Dehn fncn of the orhtoplex}
\end{proof}

\bibliographystyle{siam}
\bibliography{findehnref}

\end{document}